\documentclass[11pt]{amsart}

\usepackage{xcolor}

\usepackage{amsfonts}
\usepackage{amsmath}
\usepackage{amssymb}

\newcommand{\R}{{\mathbb R}}
\newcommand{\N}{{\mathbb N}}

\newcommand{\HH}{\mathcal{H}}
\newtheorem{theorem}{Theorem}[section]
\newtheorem{lemma}[theorem]{Lemma}
\newtheorem{prop}[theorem]{Proposition}
\newtheorem{coro}[theorem]{Corollary}
\newtheorem{claim}[theorem]{Claim}

\newtheorem{example}[theorem]{Example}
\newtheorem{remark}[theorem]{Remark}

\title{The $L_p$ centro-sectional Minkowski problem}
\author{K\'aroly J. B\"or\"oczky}
\address{Alfr\'ed R\'enyi Institute of Mathematics, Hungarian Academy
  of Sciences, Re\'altanoda u. 13-15, H-1053 Budapest, Hungary, and
Institute of Mathematics, E\"otv\"os University, P\'azm\'any P\'eter s\'et\'any 1/c, H-1117, Budapest, Hungary}
\email{boroczky.karoly.j@renyi.hu}

\author{Jiaqian Liu}
\address{School of Mathematics and statistics, Henan University, Jinming Avenue, 475001, Kaifeng, China}
\email{liujiaqian@henu.edu.cn}

\author{Renqing You}
\address{Department of Mathematics, Shanghai University, Shanghai, 200444,China}
\email{yrenqing2020@163.com}


\begin{document}
\maketitle

\begin{abstract}
As part of Lutwak's broadening of the Brunn-Minkowski theory, 
and extending the notion of affine quermassintegrals and dual curvature measure (see Milman, Yehudayoff \cite{MiY23} and Huang, Lutwak, Yang and Zhang \cite{HLYZ16}), centro-sectional measures with parameter $q\in\R$  have been recently introduced by Cai, Leng, Wu, Xi \cite{CLWX26}. In this paper, we introduce the $L_p$ centro-sectional Minkowski problem analogously to the  $L_p$ dual Minkowski problem formulated by Lutwak, Yang and Zhang \cite{LYZ18}.   
We solve the $L_{p}$ centro-sectional Minkowski problem for $p>1$ and $q>0$,  discuss the regularity and uniqueness of the solution, and prove $L_p$ Brunn-Minkowski-type inequalities when $p$ is relatively large. 
\end{abstract}

\section{Introduction}
Following the footsteps of Lutwak, Yang and Zhang \cite{LYZ18} in the case of $L_p$ dual curvature measures, we introduce the $L_p$-centro-sectional measures with parameter $q\in\R$, extending the notion of affine quermassintegrals and dual curvature measure, and in general, centro-sectional measures, see the papers Milman, Yehudayoff \cite{MiY23}, Huang, Lutwak, Yang and Zhang \cite{HLYZ16} and Cai, Leng, Wu, Xi \cite{CLWX26}. 

 We write $o$ to denote the origin in $\R^n$, $\langle\cdot,\cdot\rangle$ for the standard inner product, and $\|\cdot\|$ for its induced norm. We denote the unit ball  by $B^n=\{x\in\R^n:\|x\|\leq 1\}$, the unit sphere by $S^{n-1}=\partial B^n$. The notation $\HH^k(\cdot)$ stands for the $k$-dimensional Hausdorff measure where $\HH^{k}(\emptyset)=0$, and for the $n$-dimensional volume (Lebesgue measure) we use the notation $|\cdot|$. In particular, the volume of the unit ball is $\omega_n=|B^n|=\frac{\pi^{\frac n2}}{\Gamma(\frac n2 +1)}$ and its surface area is  $\HH^{n-1}(S^{n-1})=n\omega_n$, where $\Gamma$ is Euler's gamma function (cf. Artin \cite{Art64}). We call a compact convex set $K\subset\R^n$ with non-empty interior a convex body, and write ${\rm conv}\,X$ to denote the convex hull of an $X\subset\R^n$.  We write $\mathcal{K}^n_o$ to denote the family of compact convex sets $K\subset \R^n$ containing the origin, and $\mathcal{K}^n_{(o)}$ to denote the family of all convex bodies $K\subset\R^n$ which contain $o$ in their interior. For properties of compact convex sets we need,  we refer  to Gruber \cite{Gru07}, Schneider \cite{Sch14} and B\"or\"oczky, Figalli, Ramos \cite{BFR26}. 

For a convex compact set $K\subset\R^n$, the support function $h_K(u):S^{n-1}\to \R$ is defined as $h_K(u)=\max \{\langle x,u\rangle : x\in K\}$. For a $u\in S^{n-1}$, the face of $K$ with exterior unit normal $u$ is 
$F(K,u)=\{x\in K: \langle x,u\rangle=h_K(u) \}$. Assuming that $K\in\mathcal{K}^n_o$ is a convex body, the radial function $\varrho_K(u)=\max\{t\geq 0:tu\in K\}$ of $u\in S^{n-1}$ is also closely related to the support function. If $h_K(v)>0$ and $h_K$ is differentiable at a $v\in S^{n-1}$, and hence $F(K,v)=\{x\}$ for an $x\in\partial K$, then
$u=x/\|x\|$ satisfies
\begin{equation}
\label{rhoK-hK}
\varrho_K(u)=\|x\|=\sqrt{\|\nabla h_K(v)\|^2+h_K(v)^2}.
\end{equation}

If $K\subset \R^n$ is a convex body and has a unique supporting hyperplane at $x$, then we say that $x$ is a regular boundary point, and write $\nu_K(x)$ to denote the unique element of $\nu_K(\{x\})$. In addition, we write $\partial'K$ to denote the set of regular boundary points that satisfies
\begin{equation}
\label{regular-boundary}
\partial' K\mbox{ is a Borel set and \ }\HH^{n-1}(\partial K\backslash \partial' K)=0.
\end{equation}
Now, $\nu_K:\partial'K\to S^{n-1}$ is a continuous function that is usually called the spherical Gauss map, and the push forward measure of $\HH^{n-1}$ by $\nu_K$ is the  surface area measure $S_K$; namely, if $\eta\subset S^{n-1}$ is a Borel set, then
\begin{equation}
\label{SK-def}
S_K(\eta)=\HH^{n-1}\left(\nu_K^{-1}(\eta)\right). 
\end{equation}
It can be considered as the first variation of volume as it satisfies Alexandrov's variational formula (going back to Minkowski)
\begin{equation}
\label{Alexandrov}
\lim_{t\to 0^+}\frac{|K+t L|-|K|}{t}=\int_{S^{n-1}}h_L\,d S(K, \cdot)
\end{equation}
for any convex body $L\subset\R^n$.

The classical Minkowski problems asks for necessary and sufficient conditions for a Borel measure on $S^{n-1}$ to be the surface area measure of a convex body, which was solved by Minkowski \cite{Min97,Min03} in the discrete case (when the solution is a polytope) and in the "smooth" case (when the solution has $C^\infty_+$ boundary), and by Alexandrov \cite{Ale38,Ale96} in general. The associated Monge-Amp\'ere equation is
\begin{equation}
\label{Minkowsi-Problem-Monge-Ampere}
\det(\nabla^2h+hI_{n-1})=f
\end{equation} 
for some integrable $f:S^{n-1}\to[0,\infty)$ where $h$ is the restriction of the support function of the "solution" convex body with $dS_K=f\,d\HH^{n-1}$; moreover, $\nabla^2 h$ and $\nabla h$ are the spherical Hessian and gradient with respect to a moving frame. The regularity of the solution of the Minkowski problem \eqref{Minkowsi-Problem-Monge-Ampere} was intensively  investigated for almost a century by
Nirenberg \cite{Nir53}, Cheng, Yau \cite{ChY76} and Pogorelov \cite{Pog78} in the middle of the 20th century, and finally, Caffarelli \cite{Caf90a,Caf90b} settled this issue around 1990 (see also Chapter 4 in Figalli \cite{Fig17}). For the fundamental properties of Monge-Amp\`ere equations, see Figalli \cite{Fig17} and Trudinger, Wang \cite{TrW08}, and the for properties related to Minkowski-type problems, see  Huang, Yang, Zhang \cite{HYZ25} and  B\"or\"oczky, Figalli, Ramos \cite{BFR26}.

Since the pioneering work of Lutwak in the 1990's (see, for example, Lutwak \cite{Lut93}), various versions of the Minkowski problem have arisen with fundamental geometric meaning. For example, one can replace the Minkowski sum in \eqref{Alexandrov} by the $L_p$ linear combination defined by Firey \cite{Fir62} if $p>1$ and by B\"or\"oczky, Lutwak, Yang, Zhang \cite{BLYZ12} if $0< p<1$, and consider
\begin{align}
\label{Lp-sum}
\alpha\cdot K+_p\beta\cdot L=\left\{x\in\R^n:\langle x,u\rangle^p\leq \alpha h_K(u)^p+\beta h_L(u)^p\right\}
\end{align} 
for $K,L\in \mathcal{K}^n_o$ and $\alpha,\beta\geq 0$. Here 
$$
h_{\alpha\cdot K+_p\beta\cdot L}=\left(\alpha h_K^p+\beta h_L^p\right)^{\frac1p} \mbox{ \ \ \  if $p\geq 1$},
$$
and hence $L_p$-sum is the Minkowski sum if $p=1$. For $p>0$, Lutwak defined the $L_p$-surface area measure $S_p(K,\cdot)$ of a convex body $K\in \mathcal{K}^n_{(o)}$ by the variational formula (cf. B\"or\"oczky, Lutwak, Yang, Zhang \cite{BLYZ12})
\begin{equation}
\label{Lutwak}
\lim_{t\to 0^+}\frac{|K+_p t\cdot  L|-|K|}{t}=\frac1p\int_{S^{n-1}}h_L^p\,d S_p(K, \cdot)
\end{equation}
for any convex body $L\in \mathcal{K}^n_{(o)}$. According to Lutwak, the definition of the $L_p$-surface area measure $S_p(K,\cdot)$ can be extended to any $p\in\R$ by the formula (cf. B\"or\"oczky, Figalli, Ramos \cite{BFR26})
\begin{align}
\label{Lp-surface-area}
dS_p(K,\cdot)=&h_K^{1-p}\,dS_K\\
\label{Lp-surface-area-Bd}
\int_{S^{n-1}}g\,dS_p(K,\cdot)=&\int_{\partial'K}g(\nu_K(x))\langle x,\nu_K(x)\rangle^{1-p}\,d\HH^{n-1}(x)
\end{align} 
for a Borel measurable $g:S^{n-1}\to\R$ bounded or non-negative.

Fast forwarding, in this paper, we replace the volume with the recently intensively investigated quantity $\widetilde{\Psi}_{m,q}(K)$ from integral geometry (also introduced by Lutwak) for a convex body $K\in \mathcal{K}^n_o$, $m=1,\ldots,n-1$ and $q\in\R$ (see Milman, Yehudayoff \cite{MiY23} and Cai, Leng, Wu, Xi \cite{CLWX25,CLWX26}). In particular, writing $\nu_{m,n}$ to denote the Haar probability measure on the Grassmannian ${\rm G}(n,m)$ of linear $m$-planes of $\R^n$, if $m\in\{1,\ldots,n-1\}$ and $q\neq 0$, then
\begin{align}
\label{Psimqdef-sectionarea}
\widetilde{\Psi}_{m,q}(K)=&\int_{{\rm G}(n,m)}\HH^m(\xi\cap K)^q\,d\nu_{n,m}(\xi)\\
\label{Psimqdef-Radon}
=& \int_{{\rm G}(n,m)} \mathcal{R}_m\left(\frac1m\,\varrho_K^m\right)^q\,d\nu_{n,m}
\end{align}
where $\mathcal{R}_m (f)$ is the Radon transform on ${\rm G}(n,m)$ of  a bounded or non-negative Borel function $f:S^{n-1}\to \R$   defined by
\begin{equation}
\label{Radon-def}
\mathcal{R}_m(f)(\xi)=\int_{\xi\cap S^{n-1}}f\,d\HH^{m-1}.
\end{equation}
In addition, if $q=0$, then
$$
\widetilde{\Psi}_{m,0}(K)=\int_{{\rm G}(n,m)}\log \HH^m(\xi\cap K)\,d\nu_{n,m}(\xi).
$$
Readily, $\widetilde{\Psi}_{m,q}(K)<\infty$ if $K\in\mathcal{K}^n_{(o)}$ and $q\in \R$, or $K\in\mathcal{K}^n_{o}$ and $q>0$.  We observe that for $L\in\mathcal{K}^n_{o}$ with $K\subsetneq L$, we have $\widetilde{\Psi}_{m,q}(K)<\widetilde{\Psi}_{m,q}(L)$ if $q>0$, and $\widetilde{\Psi}_{m,q}(K)>\widetilde{\Psi}_{m,q}(L)$ if $q<0$ and $\widetilde{\Psi}_{m,q}(K)<\infty$.

For $p>0$, $q\in\R\backslash\{0\}$ and $m=1,\ldots,n-1$, $K\in \mathcal{K}^n_{(o)}$, we define the $L_p$ centro-sectional measure $\widetilde{A}_{m,q,p}(K, \cdot)$ on $S^{n-1}$ by the formula
\begin{equation}
\label{Amqp}
\lim_{t\to 0^+}\frac{\widetilde{\Psi}_{m,q}(K+_p t\cdot  L)-\widetilde{\Psi}_{m,q}(K)}{t}=\frac{q}p\int_{S^{n-1}}h_L^p\,d \widetilde{A}_{m,q,p}(K, \cdot)
\end{equation}
for any $L\in \mathcal{K}^n_{(o)}$, and if $q=0$, then
\begin{equation}
\label{Am0p}
\lim_{t\to 0^+}\frac{\widetilde{\Psi}_{m,0}(K+_p t\cdot  L)-\widetilde{\Psi}_{m,0}(K)}{t}=\frac{1}p\int_{S^{n-1}}h_L^p\,d \widetilde{A}_{m,q,0}(K, \cdot).
\end{equation}
 We establish the existence of the Borel measure $\widetilde{A}_{m,q,p}(K, \cdot)$ on $S^{n-1}$ based on the results in Cai, Leng, Wu, Xi \cite{CLWX26} (cf. Remark~\ref{Amqp-Lpsum}). 

For the general formula describing $\widetilde{A}_{m,q,p}(K, \cdot)$, we write ${\rm G}(V,k)$ to denote the Grassmanian of linear $k$-subspaces of a linear subspace $V\subset\R^n$, $0\leq k<{\rm dim}\,V$, and $\nu_{V,k}$ to denote the corresponding Haar probability measure. If $u\in S^{n-1}$ and $m=1,\ldots,n-1$, then the dual Radon transform $\mathcal{R}^*_m:S^{n-1}\to\R$ of a Borel function $F:{\rm G}(n,m)\to\R$ that is bounded or non-negative is defined by the formula
\begin{equation}
\label{dualRadon-def} 
\mathcal{R}^*_mF(u)=\frac{m\omega_m}{n\omega_n}\int_{{\rm G}(u^\bot,m-1)}F\left(\zeta+\R u\right)\,d\nu_{u^\bot,m-1}(\zeta).
\end{equation}
For a convex body $K\in \mathcal{K}^n_{o}$ and $p,q\in\R$, we use the notation $\HH^m\left(K\cap \cdot\right)^{q-1}$ for the function $\HH^m\left(K\cap \xi\right)^{q-1}$ of $\xi\in {\rm G}(n,m)$, and define the $L_p$ centro-sectional measure 
$\widetilde{A}_{m,q,p}(K,\cdot)$ on $S^{n-1}$ by the formula
 (see Sections~\ref{seccentrosectional} and \ref{secLpcentrosectional})
\begin{align}
\label{tildeAmqp-bdK-eq0}
\int_{S^{n-1}}g\,d\widetilde{A}_{m,q,p}(K,\cdot)=
\int_{\partial' K} &g(\nu_K(x))\langle \nu_K(x),x\rangle^{1-p}\|x\|^{m-n}\\
\nonumber 
&\mathcal{R}^*_m\HH^m\left(K\cap \cdot\right)^{q-1}\left(\frac{x}{\|x\|}\right)\,d\HH^{n-1}(x) 
\end{align}
 for any Borel function $g:\,S^{n-1}\to[0,\infty)$. We note that the formula \eqref{tildeAmqp-bdK-eq0} makes sense if either $p\leq 1$, or $p> 1$ and $\HH^{n-1}(\Xi_K)=0$ (for example, $K\in \mathcal{K}^n_{(o)}$)
where 
$$
\Xi_K=\{x\in\partial'K:\langle x,\nu_K(x)\rangle=0\},
$$
and hence $\Xi_K=\emptyset$ if $K\in \mathcal{K}^n_{(o)}$.
In particular, 
\begin{equation}
\label{Amqp-hAmq}
d\widetilde{A}_{m,q,p}(K,\cdot)=h_K^{-p}\,d\widetilde{A}_{m,q}(K,\cdot),
\end{equation}
where the Borel measure $\widetilde{A}_{m,q}$ on $S^{n-1}$ is the centro-sectional measure defined by  Cai, Leng, Wu, Xi \cite{CLWX26} satisfying that if $\eta\subset S^{n-1}$ is Borel, then 
\begin{align}
\label{Amq-eta} 
\widetilde{A}_{m,q}(K,\eta)=&\int_{\alpha^*_K(\eta)}\varrho_K(u)^m
\mathcal{R}^*_m\HH^m\left(K\cap \cdot\right)^{q-1}\left(u\right)\,d\HH^{n-1}(u)\\
\label{alpha*K-eta}
\alpha^*_K(\eta)=&\left\{\frac{x}{\|x\|}:x\in\partial'K\backslash \Xi_K\mbox{ and }\nu_K(x)\in\eta\right\}\subset S^{n-1}.
\end{align}
Since $\nu_K$ is continous on the Borel set $\partial'K$, we deduce that  $\alpha^*_K(\eta)$ is Borel if $\eta\subset S^{n-1}$ is so. We observe that $\widetilde{A}_{m,q,0}(K,\eta)=\widetilde{A}_{m,q}(K,\eta)$ by definition \eqref{Amqp-hAmq}, and if $o\in\partial K$ and $p\in\R$, then $\widetilde{A}_{m,q,p}(K,\{h_K=0\})=\widetilde{A}_{m,q}(K,\{h_K=0\})=0$.

It follows from Cai, Leng, Wu, Xi \cite{CLWX26} that if $K\in \mathcal{K}^n_{o}$ is a convex body, $q\neq 0$ and $\widetilde{\Psi}_{m,q}(K)$   is finite, then (see Section~\ref{seccentrosectional})
\begin{equation}
\label{Amq-sn-1}
\widetilde{A}_{m,q}(K,S^{n-1})=m\widetilde{\Psi}_{m,q}(K)<\infty.
\end{equation}
Readily, both $\widetilde{\Psi}_{m,q}(K)$  and $\widetilde{A}_{m,q,p}(K, \cdot)$ are finite for any $p,q\in\R$ and $K\in \mathcal{K}^n_{(o)}$.\\

\noindent{\bf $L_p$ centro-sectional Minkowski problem.}{\it Given  $m\in\{1,\ldots,n-1\}$ and $p,q\in\R$, characterize $\widetilde{A}_{m,q,p}(K, \cdot)$ among a meaningful family of convex bodies $K\in \mathcal{K}^n_{o}$.}\\

It follows from \eqref{rhoK-hK}, \eqref{Minkowsi-Problem-Monge-Ampere} and \eqref{tildeAmqp-bdK-eq0} that the associated  Monge-Amp\'ere equation is
\begin{align}
\label{Amqp-Monge-Ampere}
\mathcal{R}^*_m\circ\mathcal{R}_m\left(\frac{\left(\|\nabla h\|^2+h^2\right)^{\frac{m}2}}m\right)^{q-1} \left(\frac{\nabla h+h\cdot {\rm id}}{\left(\|\nabla h\|^2+h^2\right)^{\frac12}}\right)\cdot & \\
\nonumber
\cdot h^{1-p}\left(\|\nabla h\|^2+h^2\right)^{\frac{m-n}2}\cdot \det(\nabla^2h+hI_{n-1})&=f
\end{align} 
More precisely, if $p>1$, then $h$ might be zero even if $f$ is positive and continuous (see Remark~\ref{h-zero}); therefore, in line with \eqref{Amqp-hAmq} and Chou, Wang \cite{ChW06}, if $p>1$, then the right form of the Monge-Amp\'ere equation is
\begin{align}
\label{Amqp-Monge-Ampere-plarge}
\mathcal{R}^*_m\circ\mathcal{R}_m\left(\frac{\left(\|\nabla h\|^2+h^2\right)^{\frac{m}2}}m\right)^{q-1} \left(\frac{\nabla h+h\cdot {\rm id}}{\left(\|\nabla h\|^2+h^2\right)^{\frac12}}\right)\cdot & \\
\nonumber
\cdot h\cdot \left(\|\nabla h\|^2+h^2\right)^{\frac{m-n}2}\cdot \det(\nabla^2h+hI_{n-1})&=h^{p}f.
\end{align} 

We note that if $m=1$ and $q\neq 0$, then $\widetilde{V}_{q}(K)=\int_{S^{n-1}}\varrho_K(u)^q\,d\HH^{n-1}$ is the dual intrinsic volume introduced by Lutwak for $K\in\mathcal{K}^n_{o}$, and $\widetilde{C}_{q}(K, \cdot)$ is the celebrated dual curvature measure associated to $\widetilde{V}_{q}(K)$ from Huang, Lutwak, Yang, Zhang \cite{HLYZ16}; moreover, the study of the $L_p$ dual curvature measure $d\widetilde{C}_{q,p}(K, \cdot)=h_K^{-p}d\widetilde{C}_{p}(K, \cdot)$ was initiated by Lutwak, Yang, Zhang \cite{LYZ18}. 
If $K$ is $o$-symmetric ($K=-K$), then
$\widetilde{\Psi}_{1,q}(K)=\frac{2^q}{n\omega_n}\,\widetilde{V}_{q}(K)$ and $\widetilde{A}_{1,q}(K, \cdot)=\frac{2^q}{n\omega_n}\,\widetilde{C}_{q}(K, \cdot)$ (see Section~\ref{seccentrosectional}), and hence the even centro-sectional measure is a direct generalization of the even dual curvature measure. However, while $\widetilde{A}_{1,q}(K,\cdot)\geq \frac{2}{n\omega_n}\,\widetilde{C}_{q}(K,\cdot)$ if $q\geq 1$ and $\widetilde{A}_{1,q}(K,\cdot)\leq \frac{2}{n\omega_n}\,\widetilde{C}_{q}(K,\cdot)$ if $q\leq 1$, $\widetilde{A}_{1,q}(K,\cdot)$ can't be estimated from above in terms of $\widetilde{C}_{q}(K,\cdot)$ if $K\in\mathcal{K}^n_{(o)}$ is not $o$-symmetric and $q>1$.

Combining the papers B\"or\"oczky, Fodor \cite{BoF19}, Huang, Zhao \cite{HuZ18}, Chen, Li \cite{ChL21}, Lu, Pu \cite{LuP21} shows that if $p>0$ and $q\neq p,0$, then any finite Borel measure
not concentrated on a closed hemisphere is a  $L_p$ $q$th dual Minkowski curvature measure. 
See also Guang, Li, Wang \cite{GLW23} for a flow approach when $p<0$ and $q>n$, and Chou, Wang \cite{ChW06}, Bianchi,  B\"or\"oczky, Colesanti, Yang \cite{BBCY19} and Guang,  Li, Wang \cite{GLW26} when $p<0$ and $q=n$ under regularity assumptions.

In this paper, we solve the Monge-Amp\`ere equation \eqref{Amqp-Monge-Ampere-plarge} when $p>1$ and $q>0$. We note that if $m=1$, $q=n$ and $1<p<n$ - that is the case of the $L_p$ Minkowski problem -, then there exists a positive and $C^{0,\alpha}$ function  $f$ such that the unique solution of \eqref{Amqp-Monge-Ampere} corresponds to a convex body $K\in\mathcal{K}^n_o$ with $o\in\partial K$ (see Hug, Lutwak, Yang, Zhang \cite{HLYZ05}); therefore, we should allow the possibility that the origin is on the boundary.

\begin{theorem}
\label{main}
Let $p>1$, $q>0$ with $p\neq mq$, let $m\in\{1,\ldots,n-1\}$ and let $\mu$ be a finite Borel measure on $S^{n-1}$. There exists a convex body $K\in\mathcal{K}^n_{o}$ 
with $\HH^{n-1}(\Xi_K)=0$ and $\widetilde{A}_{m,q,p}(K,\cdot)<\infty$ such that 
$d\widetilde{A}_{m,q}(K,\cdot)=h_K^p\,d\mu$ if and only if $\mu$ is not concentrated on any closed hemisphere. 

If, in addition, either $m\geq 2$ and $p\geq m$, or $m=1$ and $p\geq q$, or $\mu$ is discrete ($p>1$, $q>0$ and the solution $K$ is a polytope), then any solution $K$ satisfies $K\in\mathcal{K}^n_{(o)}$ and $\widetilde{A}_{m,q,p}(K,\cdot)=\mu$.
\end{theorem}
\noindent{\bf Remarks.}
\begin{itemize}
\item Equivalently, $h=h_K|_{S^{n-1}}$ is a weak solution of \eqref{Amqp-Monge-Ampere-plarge}  if and only if $\mu$ is not concentrated on any closed hemisphere where $h$ is positive provided either $m\geq 2$ and $p\geq m$, or $m=1$ and $p>q$.  
\item In fact, we prove the existence of a convex body $K_0\in\mathcal{K}^n_{o}$ with
$$
d\widetilde{A}_{m,q}(K_0,\cdot)=m\widetilde{\Psi}_{m,q}(K_0)\cdot h_{K_0}^p\,d\mu
$$
 which $K_0$ exists even if $p=mq$.

\item Even if $m=1$ and $1<p< q$, or if $m\geq 2$ and $1<p<m$, and $o\in\partial K$, the restrictions of the measures $\widetilde{A}_{m,q,p}(K,\cdot)$ and $\mu$ to the set $\{h_K>0\}$ coincide.

\item Let $m\in\{1,\ldots,n-1\}$. For $p\in\R$ and $q\in\R$, the measures $\widetilde{A}_{m,q}(K,\cdot)$ and $\widetilde{A}_{m,q,p}(K,\cdot)$ are weakly convergent on $\mathcal{K}^n_{(o)}$, and if $q>0$ and $p\leq 1$, then $\widetilde{A}_{m,q,p}(K,\cdot)$ and $\widetilde{A}_{m,q}(K,\cdot)$ are weakly convergent even on $\mathcal{K}^n_{o}$ by defining $\widetilde{A}_{m,q}(K,\cdot)\equiv 0$ when ${\rm dim}\,K\leq n-1$. 
\end{itemize} 
We note that the special case $q=n$ and $p\geq n+(n-2)m$ of Theorem~\ref{main} is proved by Lin, Wu \cite{LiW26+}.

The solution of the $L_p$ dual Minkowski problem for $\widetilde{C}_{q,p}(K,\cdot)$ is known  to be unique in some cases:
\begin{itemize}
\item if $p>q$ and $\mu$ is discrete ($K$ is a polytope) according to  Lutwak, Yang and Zhang \cite{LYZ18},
\item if $p>q$, and $\mu$ has a positive $C^{0,\alpha}$ density function $f$ for $\alpha\in(0,1]$
according to
Huang, Zhao \cite{HuZ18},
\item if $p>1$ and $q=n$ according to Hug, Lutwak, Yang, Zhang \cite{HLYZ05}.
\end{itemize}

\begin{theorem}
\label{Psimqp-uniqueness}
Let $m\in\{1,\ldots,n-1\}$ and  $q\neq 0$, and let  $p> \max\{m,mq\}$.
If $\widetilde{A}_{m,p,q}(K,\cdot)=\widetilde{A}_{m,p,q}(L,\cdot)$ holds for $K,L\in\mathcal{K}^n_{(o)}$,
 then $K=L$.
\end{theorem}

As  is common in this context, uniqueness of the  solution of the Monge-Amp\`ere equation \eqref{Amqp-Monge-Ampere-plarge} in the smooth case is interrelated with Brunn-Minkowski type inequalities. For example, if $q\geq 1$, then the discrete version of Theorem~\ref{Psimqp-uniqueness} about uniqueness combined with the variational method that we use to prove Theorem~\ref{main} yields a weaker version of Theorem~\ref{Psimqp-BM} about Brunn-Minkowski-type inequalities by continuity, which weaker version implies Theorem~\ref{Psimqp-BM} together with the characterization of equality, which in turn leads to  Theorem~\ref{Psimqp-uniqueness}. On the other hand, if $q<1$, then we prove first Theorem~\ref{Psimqp-BM}, which in turn yields  Theorem~\ref{Psimqp-uniqueness}. In particular, we prove the following $L_p$ centro-sectional Brunn-Minkowski inequalities.

\begin{theorem}
\label{Psimqp-BM}
Let $m\in\{1,\ldots,n-1\}$ and  $q\neq 0$, and let  $p> \max\{m,mq\}$.
If $K,L\in\mathcal{K}^n_{(o)}$ and $\alpha,\beta> 0$, then
\begin{equation}
\label{Psimqp-BM-qnozero-eq}
\widetilde{\Psi}_{m,q}(\alpha\cdot K+_p\beta\cdot  L)^{\frac{p}{mq}}\geq \alpha\widetilde{\Psi}_{m,q}(K)^{\frac{p}{mq}}+\beta\widetilde{\Psi}_{m,q}(L)^{\frac{p}{mq}},
\end{equation}
and equality holds if and only if $K$ and $L$ are dilates. In addition, if $q>0$, then \eqref{Psimqp-BM-qnozero-eq} holds (including the characterization of equality) for any convex bodies $K,L\in\mathcal{K}^n_{o}$.
\end{theorem}
\noindent{\bf Remarks. } 
\begin{itemize}
\item If $q=0$, $\lambda\in(0,1)$, $p\geq m$ and $K,L\in\mathcal{K}^n_{(o)}$, then
\begin{equation}
\label{Psimqp-BM-qzero-eq}
\widetilde{\Psi}_{m,0}((1-\lambda)\cdot K+_p\lambda\cdot  L)\geq (1-\lambda)\widetilde{\Psi}_{m,0}(K)+\lambda\widetilde{\Psi}_{m,0}(L),
\end{equation}
and equality holds if and only if $K=L$.

\item If $q\in \R$,  $p> \max\{m,mq\}$ and $K,L\in\mathcal{K}^n_{(o)}$, then the $L_p$ centro-sectional Minkowski inequality is
\begin{align}
\label{Lp-Minkowski-qnonzero}
\int_{S^{n-1}}\frac{h_L^p}{h_K^p}\,d\widetilde{A}_{m,q}(K,\cdot)\geq & m\widetilde{\Psi}_{m,q}(K)^{1-\frac{p}{mq}}\widetilde{\Psi}_{m,q}(L)^{\frac{p}{mq}}\mbox{ \ if }q\neq 0\\
\label{Lp-Minkowski-q0}
\int_{S^{n-1}}\frac{h_L^p}{h_K^p}\,d\widetilde{A}_{m,0}(K,\cdot)\geq & p\left(\widetilde{\Psi}_{m,0}(L)-\widetilde{\Psi}_{m,0}(K)\right)+m \mbox{ \ if }q=0,
\end{align}
and  equality holds in \eqref{Lp-Minkowski-qnonzero} if and only if $K$ and $L$ are dilates, and in \eqref{Lp-Minkowski-q0} if and only if $K=L$.

\item Let $0< p_1<p_2$. If the inequality \eqref{Psimqp-BM-qnozero-eq} or \eqref{Lp-Minkowski-qnonzero} hold for $p=p_1$ and given $m\in\{1,\ldots,n-1\}$,  $q\neq 0$, and $K,L\in\mathcal{K}^n_{(o)}$, then \eqref{Psimqp-BM-qnozero-eq} and \eqref{Lp-Minkowski-qnonzero} hold for $p=p_2$, and if addition, $p_1\geq 1$, then equality holds for $p=p_2$ if and only if $K$ and $L$ are dilates
(cf. Lemma~\ref{LpBMMequivalent} and Claim~\ref{Lp-Lprime}).

\item If $q\neq 0$ and $K$ and $L$ are $o$-symmetric, then \eqref{Psimqp-BM-qnozero-eq} and \eqref{Lp-Minkowski-qnonzero} hold even for $p>\max\{0,mq\}$
according to Theorem~\ref{Psimqp-even}. In particular, if $m=1$, $q\neq 0$, $\alpha,\beta>0$ and $p>\max\{0,q\}$, then any $o$-symmetric convex bodies $K,L\subset\R^n$ satisfy \begin{align}
\label{Lp-BM-qnonzero-Vq}
\widetilde{V}_{q}(\alpha\cdot K+_p\beta\cdot  L)^{\frac{p}{q}}\geq &\alpha\widetilde{V}_{q}(K)^{\frac{p}{q}}+\beta\widetilde{V}_{q}(L)^{\frac{p}{q}},\\
\label{Lp-Minkowski-qnonzero-Cq}
\int_{S^{n-1}}\frac{h_L^p}{h_K^p}\,d\widetilde{C}_{q}(K,\cdot)\geq &\widetilde{V}_{q}(K)^{1-\frac{p}{q}}\widetilde{V}_{q}(L)^{\frac{p}{q}}.
\end{align}
We note that \eqref{Lp-BM-qnonzero-Vq} has been proved earlier by
Xi, Zhang \cite{XiZ22} for any $K,L\in\mathcal{K}^n_{(o)}$ and $p>\max\{0,q\}$.
\end{itemize}

 Some additional $L_p$ Brunn-Minkowski-type and Minkowski-type inequalities can be found in Section~\ref{secBM}, for example, Lemma~\ref{LpBMMequivalent} provides the usual possible equivalent formulations of Brunn-Minkowski-type inequalities.

In addition, the regularity of the solution of the $L_p$ dual Minkowski is well understood based on Caffarelli \cite{Caf90a,Caf90b} (see B\"or\"oczky, Fodor \cite{BoF19} and Huang, Zhao \cite{HuZ18}).
We prove the following concerning the uniqueness and regularity of the "smooth" $L_p$ centro-sectional Minkowski Problem.

\begin{theorem}
\label{regularity-uniqueness}
Let $p,q\in\R$, $m\in\{1,\ldots,n-1\}$ and let the $f$ in the Monge-Amp\`ere equation \eqref{Amqp-Monge-Ampere} be a positive $C^{0,\alpha}$ function on $S^{n-1}$.
\begin{enumerate}
\item[(i)] Any solution $h$ of \eqref{Amqp-Monge-Ampere} is locally $C^{2,\alpha}$ on $\{h>0\}$.
\item[(ii)] If $p>mq$, then \eqref{Amqp-Monge-Ampere} has a unique positive solution.
\end{enumerate}
\end{theorem}

We note that the even $L_p$ centro-sectional Minkowski problem is a direct generalization of the even $L_p$ dual Minkowski problem proposed by Lutwak, Yang, Zhang \cite{LYZ18}.
The even centro-sectional Minkowski problem (the even $L_p$ centro-sectional Minkowski problem  for $p=0$) for $q>0$  and $m=1,\ldots,n-1$ is solved by Cai, Leng, Wu, Xi \cite{CLWX26}. As a byproduct of our methods in this paper, we deduce the $p>0$ and $q\in\R$ case. One reason why essentially the same method yields stronger results in the even case is that variational formulas only work if the origin is in the interior of the convex body, not on the boundary. While this issue can be a real headache in general (see, for example, Proposition~\ref{extremal-problem}), the property that the origin is in the interior is given free in the origin symmetric case.

\begin{theorem}
\label{symmetric}
Let $p>0$, $q\in \R$ with $p\neq mq$, let $m\in\{1,\ldots,n-1\}$ and let $\mu$ be a finite even Borel measure on $S^{n-1}$. 
 There exists an $o$-symmetric convex body $K\subset\R^n$ 
 such that 
$\widetilde{A}_{m,q,p}(K,\cdot)=\mu$ if and only if $\mu$ is not concentrated on any great subsphere. In addition, the solution is unique if $p>mq$.
\end{theorem}
\noindent{\bf Remark.} 
 Again, if $p>0$ and $p=mq$, then there exists  an $o$-symmetric convex body $K\subset\R^n$ 
 such that 
$$
\widetilde{A}_{m,p,q}(K,\cdot)=m\widetilde{\Psi}_{m,q}(K)\cdot \mu.
$$

For the exciting history of the even $L_p$ dual Minkowski problem, see, for example, Lutwak \cite{Lut93}, or Huang, Yang, Zhang \cite{HYZ25}, or B\"or\"oczky, Figalli, Ramos \cite{BFR26}. Here we discuss the validity of the $L_p$ dual Brunn-Minkowski inequality \eqref{Lp-BM-qnonzero-Vq} for $o$-symmetric convex bodies; namely, the case $m=1$. If $p=1$ and $q=n$, then \eqref{Lp-BM-qnonzero-Vq} is just the classical Brunn-Minkowski inequality. If $0<p<1$ and $q=n$, then \eqref{Lp-BM-qnonzero-Vq} is the Brunn-Minkowski conjecture by 
B\"or\"oczky, Lutwak, Yang, Zhang \cite{BLYZ12} intensively investigated in the last decade, which has been verified if $p_n<p<1$ by combined efforts of Chen, Huang, Li, Liu \cite{CHLL20} and Kolesnikov, Milman \cite{KoM22} (see also Puttermann \cite{Put21}).

If $p=1$ and $1<q<n$, then the dual Brunn-Minkowski inequality \eqref{Lp-BM-qnonzero-Vq} for $o$-symmetric convex bodies has been long conjectured by Lutwak, and the conjecture has been recently verified by Sadovsky, Zhang \cite{SaZ25}. 

Concerning the structure of the paper, Section~\ref{secintegralformulas} discusses some properties of the Radon transform and the dual Radon transform, and Section~\ref{seccentrosectional} and Section~\ref{secLpcentrosectional} introduce the centro-sectional and $L_p$ centro-sectional measures. Some fundamental uniqueness and regularity results are obtained in Section~\ref{secregularity-uniqueness} in the smooth and in the discrete case. Theorem~\ref{main} and the existence part of Theorem~\ref{symmetric} are proved in Section~\ref{secproofmain}, and finally, Section~\ref{secBM} completes the paper with the Brunn-Minkowski-type results.

\section{Some integral formulas over the Grassmannian}
\label{secintegralformulas}

In this section, we collect some useful integral  formulas for the Grassmannian ${\rm G}(m,n)$ where $m\in\{1,\ldots,n-1\}$. First, 
we write $\Pi_V$ to denote the orthogonal projection into a proper linear subspace $V\subset\R^n$. After fixing a $u\in S^{n-1}$, it is well known (see, for example, Anderson \cite{And03}, Section~5.2.2) that the distribution of the random variable $\|\Pi_\xi u\|^2$ of $\xi\in {\rm G}(m,n)$ is the Beta distribution with parameters $(\frac{m}2,\frac{n-m}2)$, and hence if $0\leq \alpha<\beta\leq 1$, then
\begin{align}
\label{Gnm-u}
&\nu_{n,m}\left(\left\{\xi\in {\rm G}(m,n):\alpha\leq \|\Pi_\xi u\|^2\leq \beta\right\}\right)\\
\nonumber
=&\frac1{B(\frac{m}2,\frac{n-m}2)}\int_\alpha^\beta t^{\frac{m}2-1}(1-t)^{\frac{n-m}2-1}\,dt
\end{align}
where $B(a,b)=\int_0^1t^{a-1}(1-t)^{b-1}\,dt$ is the Beta function.

For the second formula, if $F:{\rm G}(n,m)\to[0,\infty)$ is Borel measurable, then
\begin{equation}
\label{xi-in-hyperplanes}
\int_{{\rm G}(n,m)} Fd\nu_{n,m}=\frac1{n\omega_n}\int_{S^{n-1}}\int_{{\rm G}(u^\bot,m)} F\,d\nu_{u^\bot,m}\,d\HH^{n-1}
\end{equation}
as the integral on the right is also invariant under the action of $O(n)$. 

For the first application of \eqref{xi-in-hyperplanes}, if $w\in S^{n-1}$ and $\delta\in[0,1)$, then we define
\begin{equation}
\label{Sigma-def}
\Sigma(w,\delta)=\{u\in S^{n-1}:\langle u,w\rangle\geq\delta\}.
\end{equation}
In particular, if $\delta=\cos \alpha$ for an $\alpha\in(0,\frac{\pi}2)$, then $\Sigma(w,\delta)$ is the closed spherical cap of center $w$ and geodesic radius $\alpha$.

\begin{claim}
\label{mplanes-flatcones}
For $n\geq 3$,  $\theta\in(0,1)$ and $m\in\{2,\ldots,n-1\}$, there exists $\gamma=\gamma(\theta,n,m)$ such that if $v,w\in S^{n-1}$ with $\langle v,w\rangle=0$, then
$$
\nu_{n,m}\left(\left\{\xi\in{\rm G}(n,m):\xi\cap v^\bot\cap \Sigma(w,\theta)\neq \emptyset\right\}\right)\geq \gamma.
$$
\end{claim}
\begin{proof}
We prove the statement by induction on $n-m$. First, let $n=m+1$. Choose a $p\in S^{n-1}\cap v^\bot\cap w^\bot$. If $u\in \Sigma(p,\theta)$, then let
$q\in u^\bot\cap S^{n-1}\cap {\rm lin}\,\{p,w\}$ be with $\langle q,w\rangle>0$, and hence this $q$ satisfies $q\in \Sigma(w,\theta)$. Thus we may choose $\gamma(\theta,m+1,m)=\frac1{n\omega_n}\,\HH^{n-1}(\Sigma(p,\theta))$.

Next, let $n\geq m+2$, and let $\theta=\cos \alpha$ for $\alpha\in(0,\frac{\pi}2)$. For a $p\in S^{n-1}\cap v^\bot\cap w^\bot$, we have seen that if $u\in \Sigma(p,\cos\frac{\alpha}2)$,  then $u^\bot\cap v^\bot\cap \Sigma(w,\cos\frac{\alpha}2)\neq \emptyset$. In turn, if $b\in \Sigma(w,\cos\frac{\alpha}2)$, then the spherical triangle inequality yields that  $\Sigma(b,\cos\frac{\alpha}2)\subset \Sigma(w,\theta)$. It follows that by induction that
$$
\nu_{u^\bot,m}\left(\left\{\xi\in{\rm G}(u^\bot,m):\xi\cap v^\bot\cap \Sigma(w,\theta)\neq \emptyset\right\}\right)\geq \gamma\left(\cos\frac{\alpha}2,n-1,m\right).
$$
Therefore,  
$\gamma\left(\theta,n,m\right)=\gamma\left(\cos\frac{\alpha}2,n,n-1\right)\gamma\left(\cos\frac{\alpha}2,n-1,m\right)$ works by \eqref{xi-in-hyperplanes}. 
\end{proof}

For another application of \eqref{xi-in-hyperplanes}, we observe that if $K\in\mathcal{K}^n_o$ is a convex body, then $h_K(u)=0$ at some $u\in S^{n-1}$ if and only if $o\in\partial K$ and  $u\in N_K(o)$ for the closed convex cone ("normal cone")
\begin{equation}
\label{NKo}
N_K(o)=\{z:\langle z,y\rangle\leq 0\mbox{ for }y\in K\},
\end{equation}
and hence $\{h_K=0\}=N_K(o)\cap S^{n-1}$.
The polar of $N_K(o)$ is
$$
N_K(o)^*=\{x\in\R^n:\langle x,y\rangle\leq 0\mbox{ for any }y\in N_K(o)\}.
$$
If $K\in\mathcal{K}^n_{(o)}$, then we define $N_K(o)=\{o\}$, and hence $N_K(o)^*=\R^n$.
For any convex body $K\in \mathcal{K}^n_{o}$, we have
$$
{\rm int}\,N_K(o)^*=\left\{tx:x\in{\rm int}\,K\mbox{ and }t>0\right\}.
$$
If $K\in\mathcal{K}^n_{o}$ is a convex body with $o\in\partial K$ and $v\in S^{n-1}$, then
we call $v$ a tangent vector to $K$ if $v\in\partial N_K(o)^*$, which holds if and only if $tv\not\in{\rm int}\,K$ for $t>0$, and
\begin{equation}
\label{Ktangent-vector}
\exists \mbox{ sequences $x_k\in{\rm int}\,K$ and $\lambda_k>0$ such that }\lim_{k\to\infty}\lambda_kx_k=v.
\end{equation}
We deduce from  \eqref{Ktangent-vector} and the convexity of $N_K(o)^*$  that if $v\in S^{n-1}\cap V$ for a linear subspace $V\subset \R^n$ such that $1\leq {\rm dim}\,V\leq n-1$ with $V\cap{\rm int}\,K\neq \emptyset$, then 
\begin{equation}
\label{Ktangent-vector-V}
v \mbox{ is a tangent vector to }V\cap K\mbox{ if and only if }v \mbox{ is a tangent vector to }K.
\end{equation}
For $m\in\{1,\ldots,n-1\}$ and linear subspace $V\subset \R^n$ such that $m+1\leq {\rm dim}\,V\leq n$ and $V\cap{\rm int}\,K\neq \emptyset$, the closed subset of tangent $m$ spaces to $K$ with respect to $V$ is
\begin{align*}
\mathcal{T}_{m,K,V}=\{&\xi\in {\rm G}(V,m):\xi\cap{\rm int}\,K=\emptyset\mbox{ and }\\
\nonumber
&\exists v\in \xi\cap S^{n-1}\mbox{ tangent to }K\cap V\}.
\end{align*}
We set $\mathcal{T}_{m,K,\R^n}=\mathcal{T}_{m,K}$, and deduce from \eqref{Ktangent-vector-V} that
\begin{equation}
\label{Ktangent-space-withinV}
\mathcal{T}_{m,K,V}={\rm G}(V,m)\cap \mathcal{T}_{m,K}.
\end{equation}
Finally, for $m=n-1$, we have
\begin{equation}
\label{Ktangent-space-n-1}
\mathcal{T}_{n-1,K}=\{u^\bot:u\in S^{n-1}\cap \partial N_K(o)\}
\end{equation}
where for $u\in S^{n-1}$, 
\begin{equation}
\label{K-m-space-interior}
u^\bot\cap{\rm int}\,K\neq \emptyset 
\mbox{ if and only if } u\not\in N_K(o).
\end{equation}

\begin{lemma}
\label{numn-tangent-mspaces}
If $m\in\{1,\ldots,n-1\}$, and $K\in\mathcal{K}^n_{o}$ is a convex body with $o\in \partial K$, then
$$
\nu_{n,m}\left(\mathcal{T}_{m,K}\right)=0.
$$
\end{lemma}
\begin{proof}
We prove Lemma~\ref{numn-tangent-mspaces} by induction on $n\geq 2$ where the case $n=2$ is trivial.

When $n\geq 3$, then we apply \eqref{xi-in-hyperplanes} to the function $F=\mathbf{1}_{\mathcal{T}_{m,K}}$. If $u\in S^{n-1}$ and $u^\bot\cap {\rm int}\,K\neq\emptyset$, then induction and \eqref{Ktangent-space-withinV} yield that
\begin{equation}
\label{numn-tangent-mspaces-ubotint}
\nu_{n,m}\left(\mathcal{T}_{m,K}\cap {\rm G}(u^\bot,m)\right)=0.
\end{equation}
Since $\HH^{n-1}( S^{n-1}\cap \partial N_K(o))=0$, we conclude Lemma~\ref{numn-tangent-mspaces} by \eqref{Ktangent-space-n-1}, \eqref{K-m-space-interior} and \eqref{numn-tangent-mspaces-ubotint}.
\end{proof}

 Finally, according to Rubin \cite{Rub15} (see also  Cai, Leng, Wu, Xi \cite{CLWX25}), the Radon transform $\mathcal{R}_m$ (cf. \eqref{Radon-def})  and the dual Radon transform $\mathcal{R}_m^*$ (cf. \eqref{dualRadon-def}) satisfy the following duality relation for $m\in\{1,\ldots,n-1\}$. If   
$f\in L^\infty(S^{n-1})$ and $F\in L^\infty({\rm G}(n,m))$ are Borel measurable, then
\begin{equation}
\label{Rm-Rm*-eq0}
\int_{S^{n-1}}f\cdot \mathcal{R}_m^*F\,d\HH^{n-1}=\int_{{\rm G}(n,m)} F\cdot \mathcal{R}_mf\,d\nu_{n,m}.
\end{equation}
Here Borel measurability ensures that $\mathcal{R}_mf$ and $\mathcal{R}_m^*F$ are measurable. We note that the duality formula \eqref{Rm-Rm*-eq}, originally for continuous functions, was due to Helgason \cite{Hel64}. In our paper, we need the following version.

\begin{lemma}
\label{Rm-Rm*}
Let $m\in\{1,\ldots,n-1\}$. The Borel measurable functions $f:S^{n-1}\to\R$ and $F: {\rm G}(n,m)\to\R$ satisfy 
\begin{equation}
\label{Rm-Rm*-eq}
\int_{S^{n-1}}f\cdot \mathcal{R}_m^*F\,d\HH^{n-1}=\int_{{\rm G}(n,m)} F\cdot \mathcal{R}_mf\,d\nu_{n,m}
\end{equation}
if either
$\int_{S^{n-1}}|f|\cdot \mathcal{R}_m^*|F|\,d\HH^{n-1}<\infty$, or 
$\int_{{\rm G}(n,m)} |F|\cdot \mathcal{R}_m|f|\,d\nu_{n,m}<\infty$, and both are
 locally finite in their support in the following sense: There exist open sets $U\subset S^{n-1}$ and $\mathcal{U}\subset {\rm G}(n,m)$ with $\HH^{n-1}\left(\partial U\right)=\nu_{m,n}\left(\partial \mathcal{U}\right)=0$ such that $f(u)=F(\xi)=0$ for $u\not\in{\rm cl}\,U$ and $\xi \not\in{\rm cl}\,\mathcal{U}$, and there exist increasing sequence of compact subsets $\{Z_k\subset U\}$ and $\{\mathcal{Z}_k\subset \mathcal{U}\}$ for $k\in\N$ such that  $\bigcup_{k\in\N}Z_k=U$,  $\bigcup_{k\in\N}\mathcal{Z}_k= \mathcal{U}$ and the restrictions $f|_{Z_k}$ and $F|_{\mathcal{Z}_k}$ are bounded.
\end{lemma}
\noindent{\bf Remark.} The reason why in some cases we need Lemma~\ref{Rm-Rm*} instead of \eqref{Rm-Rm*-eq0} is that if $q\in(0,1)$ and $K\subset\R^n$ is a strictly convex body with $o\in\partial K$  (no segments in $\partial K$), then $F(\xi)=\HH^{m-1}(\xi\cap K)^{q-1}$ is not a bounded function of $\xi\in {\rm G}(m,n)$.
\begin{proof}
Writing $f$ and $F$ as the sum of the negative and the positive part, and using the linearity of the operators $\mathcal{R}_m$ and $\mathcal{R}_m^*$, we may assume that $f$ and $F$ are non-negative. 

 As $f_k=f\cdot \mathbf{1}_{Z_k}$ and $F_k=F\cdot \mathbf{1}_{\mathcal{Z}_k}$ are bounded, we deduce from \eqref{Rm-Rm*-eq0}  that
\begin{equation}
\label{Rm-Rm*-fkFk}
\int_{S^{n-1}}f_k\cdot \mathcal{R}_m^*F_k\,d\HH^{n-1}=\int_{{\rm G}(n,m)} F_k\cdot \mathcal{R}_mf_k\,d\nu_{n,m}.
\end{equation}
Here $\{f_k\cdot \mathcal{R}_m^*F_k\}$ is a monotone increasing sequence of functions tending pointwise to $f\cdot \mathcal{R}_m^*F$ except for the points of $\partial U$, and $\{F_k\cdot \mathcal{R}_mf_k\}$ is a monotone increasing sequence of functions tending pointwise to $F\cdot \mathcal{R}_mf$ except for the points of $\partial \mathcal{U}$. We deduce from \eqref{Rm-Rm*-fkFk} and Lebesgue's Dominated Convergence Theorem that if either side of  \eqref{Rm-Rm*-eq} is finite, then the two sides are equal.
\end{proof}

Let us provide a simple application of the duality formula \eqref{Rm-Rm*-eq0}.

\begin{claim}
\label{mplanes-ball}
For $m\in\{1,\ldots,n-1\}$, there exists $\gamma=\gamma(n,m)\in(0,1)$ depending on $n,m$ such that if $0<r<R$ and $\|x_0\|\leq R$ for $x_0\in\R^n$, then
\begin{equation}
\label{mplanes-ball-eq}
\nu_{n,m}\left(\left\{\xi\in{\rm G}(n,m):\xi\cap(x_0+rB^n)\neq \emptyset\right\}\right)\geq \gamma(n,m)\cdot \frac{r^{n-m}}{R^{n-m}}.
\end{equation}
\end{claim}
\begin{proof}
We prove Claim~\ref{mplanes-ball} by induction on $m\geq 1$. The statement holds if $m=1$ as the $\HH^{n-1}$ measure of a cap of $S^{n-1}$ of spherical radius $\arcsin \frac{r}{R}$ is at least $\omega_{n-1}\,\frac{r^{n-1}}{R^{n-1}}$, and hence we can choose $\gamma(n,1)=\frac{\omega_{n-1}}{n\omega_n}$.

Let $m\geq 2$. For $Z=\{\xi\in{\rm G}(n,m):\xi\cap (x_0+rB^n)\neq \emptyset\}$, \eqref{Rm-Rm*-eq0} yields that
$$
\nu_{n,m}(Z)=\frac1{m\omega_m}\int_{{\rm G}(n,m)}\mathbf{1}_Z\mathcal{R}_m\mathbf{1}_{S^{n-1}}\,d\nu_{n,m}
=\frac1{m\omega_m}
\int_{S^{n-1}}\mathcal{R}_m^*\mathbf{1}_Z\,d\HH^{n-1}.
$$
As for any $u\in S^{n-1}$, induction implies
\begin{align*}
\mathcal{R}_m^*\mathbf{1}_Z(u)=&\frac{m\omega_m}{n\omega_n}\,\nu_{u^\bot,m-1}\left(\left\{\zeta\in{\rm G}(u^\bot,m-1):\zeta\cap(\Pi_{u^\bot}x_0+rB^n)\neq \emptyset\right\}\right)\\
\geq & \frac{m\omega_m}{n\omega_n}\cdot\gamma(n-1,m-1)\,\frac{r^{(n-1)-(m-1)}}{R^{(n-1)-(m-1)}},
\end{align*}
we conclude Claim~\ref{mplanes-ball}.
\end{proof}

In Claim~\ref{mplanes-ball}, if $(x_0+\frac{r}2\,B^n)\cap \xi\neq \emptyset$ holds for $\xi\in {\rm G}(n,m)$, then $\HH^m\left((x_0+rB^n)\cap\xi\right)\geq \frac{r^m}{2^m}\,\omega_m$. Therefore, Claim~\ref{mplanes-ball} yields the following estimate.

\begin{coro}
\label{mplanes-ball-section} 
For $n\geq 3$, $m\in\{2,\ldots,n-1\}$, $\tau\in\R$ and $0<r<R$, there exists $\gamma\in(0,1)$ depending on $n,m,\tau,r,R$, such that
if $K\subset RB^n$ is a convex body with $x_0+rB^n\subset K$ for some $x_0\in K$, then 
$$
\gamma\leq \mathcal{R}^*_m\HH^m\left(K\cap \cdot\right)^\tau(u)\leq \gamma^{-1}
$$
holds for any $u\in S^{n-1}$.
\end{coro}

\section{On the centro-sectional measure}
\label{seccentrosectional}

Following  Cai, Leng, Wu, Xi \cite{CLWX26}, let $q\in\R$, $m\in\{1,\ldots,n-1\}$ and $K\in \mathcal{K}^n_{o}$ be a convex body, and let us define the Borel centro-sectional measure $\widetilde{A}_{m,q}(K,\eta)$ on $S^{n-1}$ in a way such that
 if $\eta\subset S^{n-1}$ is Borel, then
\begin{align}
\label{Amq-eta-def}
\widetilde{A}_{m,q}(K,\eta)=&\int_{\alpha^*_K(\eta)}\varrho_K(u)^m
\mathcal{R}^*_m\HH^m\left(K\cap \cdot\right)^{q-1}\left(u\right)\,d\HH^{n-1}(u)\\
\label{alpha*K-eta-def}
\alpha^*_K(\eta)=&\left\{\frac{x}{\|x\|}:x\in\partial'K\backslash \Xi_K\mbox{ and }\nu_K(x)\in\eta\right\}\subset S^{n-1}.
\end{align}
Since $\nu_K$ is continous on the Borel set $\partial'K$, we deduce that  $\alpha^*_K(\eta)$ is Borel if $\eta\subset S^{n-1}$ is so. We observe that $\widetilde{A}_{m,q}(K,\cdot)$ is $mq$-homogenous; namely, if $\lambda>0$, then
\begin{equation}
\label{Amq-homogeneous}
\widetilde{A}_{m,q}(\lambda\,K,\cdot)=\lambda^{mq}\widetilde{A}_{m,q}(K,\cdot),
\end{equation}
and if $o\in\partial K$, then 
\begin{equation}
\label{Amq-hK0}
\widetilde{A}_{m,q}(K,\{h_K=0\})=\widetilde{A}_{m,q}(K,N_K(o)\cap S^{n-1})=0.
\end{equation}

\begin{example}[Centro-sectional measure of polytopes]
\label{example-Amq-polytope}
For $q\in\R$ and $m\in\{1,\ldots,n-1\}$, let $P\in \mathcal{K}^n_{(o)}$ be polytope with exterior unit normals $u_1,\ldots,u_k$ and facets $F_i=F_P(u_i)$ for $i=1,\ldots,k$. For $i=1,\ldots,k$, if $F'_i=\alpha_P^*(\{u_i\})$ is the radial projection of $F_i$ onto $S^{n-1}$, then ${\rm supp}\,\widetilde{A}_{m,q}(P,\cdot)=\{u_1,\ldots,u_k\}$ and
\begin{equation}
\label{Polytope-Amq-eq}
\widetilde{A}_{m,q}(P,\{u_i\})=\int_{F'_i}\varrho_P^m\,
\mathcal{R}^*_m\HH^m\left(P\cap \cdot\right)^{q-1}\,d\HH^{n-1}.
\end{equation}
\end{example}

We note that if $m=1$ and $q\neq 0$, then $\widetilde{V}_{q}(K)=\int_{S^{n-1}}\varrho_K(u)^q\,d\HH^{n-1}$ is the dual intrinsic volume introduced by Lutwak for $K\in\mathcal{K}^n_{o}$, and the associated dual curvature measure $\widetilde{C}_{q}(K, \cdot)$ is defined by
\begin{equation}
\label{Cq-eta-def}
\widetilde{C}_{q}(K,\eta)=\int_{\alpha^*_K(\eta)}\varrho_K(u)^q\,d\HH^{n-1}
\end{equation}
for Borel $\eta\subset S^{n-1}$ by Huang, Lutwak, Yang, Zhang \cite{HLYZ16}. 
On the other hand, \eqref{Amq-eta-def} yields that 
\begin{equation}
\label{A1q-eta-def}
\widetilde{A}_{1,q}(K,\eta)=\frac2{n\omega_n}\int_{\alpha^*_K(\eta)}\varrho_K(u)\left(\varrho_K(u)+\varrho_K(-u)\right)^{q-1}\,d\HH^{n-1}.
\end{equation}
Thus, for any convex body $K\in\mathcal{K}^n_{o}$, we have
\begin{align}
\label{A1q-Cq-est}
\widetilde{A}_{1,q}(K, \cdot)\geq &\frac{2}{n\omega_n}\,\widetilde{C}_{q}(K, \cdot) \mbox{ if }q\geq 1\mbox{ \ and \ } \widetilde{A}_{1,q}(K, \cdot)\leq \frac{2}{n\omega_n}\,\widetilde{C}_{q}(K, \cdot) \mbox{ if }q\leq 1;\\
\label{A1q-Cq-symm}
\widetilde{A}_{1,q}(K, \cdot)=&\frac{2^q}{n\omega_n}\,\widetilde{C}_{q}(K, \cdot)\mbox{ if $K=-K$;}\\
\label{A1q-Cq-bd}
\widetilde{A}_{1,q}(K, \cdot)=&\frac{2}{n\omega_n}\,\widetilde{C}_{q}(K, \cdot)\mbox{ if $o\in{\rm bd}\,K$}
\end{align}
as $\varrho_K(-\nu_K(x))=0$ if $x\in\partial'K\backslash \Xi_K$ (cf. \eqref{alpha*K-eta-def}). However, $\widetilde{A}_{1,q}(K,\eta)$ might be much larger than  $\widetilde{C}_{q}(K,\eta)$ for some Borel set $\eta\subset S^{n-1}$ if $K\in\mathcal{K}^n_{(o)}$ is not $o$-symmetric and $q>1$.

For $X\subset \R^n$, we write 
$$
{\rm pos}\,X=\{tx:x\in X\mbox{ and }t\geq 0\},
$$
that is a closed set if $X$ is compact and $o\not\in{\rm conv}\,X$.
In order to represent $\widetilde{A}_{m,q}(K, \cdot)$ as a boundary integral for a convex body $K\in \mathcal{K}^n_{o}$, if $k$ is large enough, then let
\begin{equation}
\label{CKk-def}
C_{k,K}={\rm pos}\{x\in S^{n-1}\cap N_K(o)^*:x+\mbox{$\frac1k$}\,B^n\subset N_K(o)^*\},
\end{equation}
where $C_{k,K}$ is a closed convex cone, and $C_{k,K}=N_K(o)^*=\R^n$ in the case $o\in{\rm int}\,K$.
For $r_{k,K}=\frac12\min\{\varrho_K(u):u\in S^{n-1}\cap C_{k,K}\}$, we have $r_{k,K}S^{n-1}\cap C_{k,K}\subset{\rm int}\,K$, and hence there exists $\sigma_{k,K}>0$ such that
\begin{equation}
\label{sigmak-inCkK}
\left(r_{k,K}S^{n-1}\cap C_{k,K}\right)+\sigma_{k,K}B^n\subset K.
\end{equation}
In addition, we define
\begin{align*}
U_K=&S^{n-1}\cap{\rm int}\,N_K(o)^*\\
\mathcal{U}_K=&\{\xi\in G(n,m):\xi\cap {\rm int}\,K\neq\emptyset\}=
\{\xi\in G(n,m):\xi\cap U_K\neq\emptyset\}.
\end{align*}
As the radial projection from $\partial'K\cap{\rm int}\,N_K(o)^*$ to $S^{n-1}\cap{\rm int}\,N_K(o)^*$ is locally Lipschitz, it follows that 
\begin{equation}
\label{ThetaK-0}
\HH^{n-1}\left(U_K\backslash \Theta_K\right)=0
\end{equation}
holds for the Borel set 
\begin{equation}
\label{ThetaK-def}
\Theta_K=\left\{\frac{x}{\|x\|}:x\in \partial'K\cap{\rm int}\,N_K(o)^*\right\}\subset U_K.
\end{equation}
We deduce from \eqref{Rm-Rm*-eq}  that if a Borel set $\Upsilon\subset U_K$ satisfies $\HH^{n-1}\left(U_K\backslash \Upsilon\right)=0$, then
\begin{align*}
\int_{\mathcal{U}_K}\HH^{m-1}(\Upsilon\cap\xi)\,d\nu_{m,n}(\xi)=&
\int_{{\rm G}(n,m)} \mathbf{1}_{\mathcal{U}_K}\cdot \mathcal{R}_m \mathbf{1}_{\Upsilon}\,d\nu_{n,m}\\
=
\int_{S^{n-1}} \mathcal{R}_m^*\mathbf{1}_{\mathcal{U}_K}\cdot \mathbf{1}_{\Upsilon}\,d\HH^{n-1}
=&\int_{S^{n-1}} \mathcal{R}_m^*\mathbf{1}_{\mathcal{U}_K}\cdot \mathbf{1}_{U_K}\,d\HH^{n-1}\\
=\int_{{\rm G}(n,m)} \mathbf{1}_{\mathcal{U}_K}\cdot \mathcal{R}_m \mathbf{1}_{U_K}\,d\nu_{n,m}=&\int_{\mathcal{U}_K}\HH^{m-1}(U_K\cap\xi)\,d\nu_{m,n}(\xi);
\end{align*}
therefore, 
\begin{equation}
\label{U-Upsilon-section}
\nu_{m,n}\left(\left\{\xi\in \mathcal{U}_K:\HH^{m-1}(\Upsilon\cap\xi)=\HH^{m-1}(U_K\cap\xi)\right\}\right)=\nu_{m,n}\left(\mathcal{U}_K\right).
\end{equation}
Next, we define the Borel function $\alpha_K:S^{n-1}\to S^{n-1}$ by the formula
$$
\alpha_K(u)=
\left\{\begin{array}{rl}
\nu_K(\varrho_K(u)\cdot u)&\mbox{ if }u\in \Theta_K\\
e_0&\mbox{ if }u\in S^{n-1}\backslash\Theta_K
\end{array}\right. 
$$
where $e_0\in S^{n-1}$ is a fixed unit vector.

\begin{lemma}
\label{Amq-sn-1-lemma}
If $m\in\{1,\ldots,n-1\}$, $q\neq 0$ and $K\in \mathcal{K}^n_{o}$ is a convex body such that $\widetilde{\Psi}_{m,q}(K)<\infty$ (that is always the case if $q>0$ or $K\in \mathcal{K}^n_{(o)}$), then 
\begin{enumerate}
\item[(i)] for a bounded Borel function $g:S^{n-1}\to\R$, we have
$$
\int_{S^{n-1}}g\,d\widetilde{A}_{m,q}(K,\cdot)=\int_{{\rm G}(n,m)}\HH^m\left(K\cap \cdot\right)^{q-1}\mathcal{R}_m\left(\varrho_K^m\cdot g\circ\alpha_K\right)
\,d\nu_{n,m};
$$

\item[(ii)] $\widetilde{A}_{m,q}(K,S^{n-1})=m\widetilde{\Psi}_{m,q}(K)<\infty$.
\end{enumerate}
\end{lemma}
\noindent{\bf Remark.} Similar argument yields that if $m\in\{1,\ldots,n-1\}$,  $K\in \mathcal{K}^n_{o}$ is a convex body and $g:S^{n-1}\to\R$ is a bounded Borel function, then
\begin{align*}
\int_{S^{n-1}}g\,d\widetilde{A}_{m,0}(K,\cdot)=&\int_{{\rm G}(n,m)}\HH^m\left(K\cap \cdot\right)^{-1}\mathcal{R}_m\left(\varrho_K^m\cdot g\circ\alpha_K\right)
\,d\nu_{n,m}; \\
\widetilde{A}_{m,0}(K,S^{n-1})=&m.
\end{align*}

\begin{proof}
For (i), we plan to apply Lemma~\ref{Rm-Rm*} where $f=\varrho_K^m\cdot g\circ\alpha_K$, and $F(\xi)=\HH^m(K\cap\xi)^{q-1}$ for $\xi\in G(n,m)$. If $o\in{\rm int}\,K$,  then let $U=S^{n-1}$ and $\mathcal{U}=G(n,m)$, and if $o\in\partial K$, then let  $U=S^{n-1}\cap{\rm int}\,N_K(o)^*$ and $\mathcal{U}=\{\xi\in G(n,m):\xi\cap {\rm int}\,K\neq\emptyset\}$. 
Then $\nu_{m,n}(\partial \mathcal{U})=0$ by Lemma~\ref{numn-tangent-mspaces}, and
we readily have  $\HH^{n-1}(\partial U)=0$  and $f(u)=F(\xi)=0$ if $u\not \in{\rm cl}\,U$ or $\xi\not\in {\rm cl}\mathcal{U}$.
Since there exists $M>0$ such that $|g(u)|\leq M$ for $u\in S^{n-1}$, we deduce that
\begin{align*}
\int_{{\rm G}(n,m)}|F|\cdot \mathcal{R}_m|f|\,d\nu_{n,m}\leq &
M\int_{{\rm G}(n,m)}\HH^m\left(K\cap \cdot\right)^{q-1}\mathcal{R}_m\varrho_K^m\,d\nu_{n,m}\\
=&
Mm\widetilde{\Psi}_{m,q}(K)<\infty.
\end{align*}
For large $k$, we consider $Z_k=C_{k,K}\cap S^{n-1}$ and  $\mathcal{Z}_k=\{\xi\in G(n,m):\xi\cap C_{k,K}\neq\{o\}\}$ that satisfy the conditions of Lemma~\ref{Rm-Rm*} by \eqref{sigmak-inCkK}.
It follows from Lemma~\ref{Rm-Rm*} that
\begin{align*}
\int_{S^{n-1}}g\,d\widetilde{A}_{m,q}(K,\cdot)=&
\int_{U}(g\circ\alpha_K)\varrho_K^m \cdot
\mathcal{R}^*_m\HH^m\left(K\cap \cdot\right)^{q-1}\,d\HH^{n-1}\\
=&\int_{S^{n-1}}f\cdot \mathcal{R}_m^*F\,d\HH^{n-1}=
\int_{{\rm G}(n,m)} F\cdot \mathcal{R}_mf\,d\nu_{n,m},
\end{align*}
proving (i).

For (ii), we take $g$ to be the constant $1$ function, and hence (i) yields that
\begin{align*}
\widetilde{A}_{m,q}(K,S^{n-1})=&\int_{{\rm G}(n,m)}\HH^m\left(K\cap \cdot\right)^{q-1}\mathcal{R}_m\varrho_K^m\,d\nu_{n,m}\\
=&m\int_{{\rm G}(n,m)}\HH^m\left(K\cap \cdot\right)^q\,d\nu_{n,m}=
m\widetilde{\Psi}_{m,q}(K).
\end{align*}
\end{proof}

Next, we consider the representation of $\widetilde{A}_{m,q}(K,\cdot)$ as a boundary integral. 
According to Huang, Lutwak, Yang and Zhang \cite{HLYZ16} and Corollary 2.6.7 in B\"or\"oczky, Figalli, Ramos \cite{BFR26}, if   $K\in \mathcal{K}^n_{o}$ is a convex body and $\varphi:\partial K\to[0,\infty)$ is Borel, then 
$$
\int_{S^{n-1}}\varphi(\varrho_K(u)\cdot u)\varrho_K(u)^n\,d\HH^{n-1}(u)=\int_{\partial' K} \varphi(x)\langle \nu_K(x),x\rangle \,d\HH^{n-1}(x).
$$
Therefore, if $m\in\{1,\ldots,n-1\}$, $q\in\R$ and $g:S^{n-1}\to\R$  is a Borel measurable function that is  bounded or non-negative, then
\begin{align}
\label{tildeAmq-bdK-eq}
\int_{S^{n-1}}g\,d\widetilde{A}_{m,q}(K,\cdot)=
\int_{\partial' K} &g(\nu_K(x))\, \|x\|^{m-n}\langle \nu_K(x),x\rangle\\
\nonumber 
&\mathcal{R}^*_m\HH^m\left(K\cap \cdot\right)^{q-1}\left(\frac{x}{\|x\|}\right)\,d\HH^{n-1}(x). 
\end{align}
Applying \eqref{tildeAmq-bdK-eq} to any open subset $U\subset S^{n-1}$ yields  that if $m\in\{1,\ldots,n-1\}$, $q\in\R$  and $K\in \mathcal{K}^n_{o}$ is a convex body, then (cf. \eqref{SK-def})
\begin{equation}
\label{suppAmqKSK}
{\rm supp}\,\widetilde{A}_{m,q}(K,\cdot)\subset {\rm supp}\,S_K.
\end{equation}
We note that 
\begin{equation}
\label{suppSK-regular-points}
{\rm supp}\,S_K={\rm cl}\{\nu_K(x):x\in\partial'K\}
\end{equation}
 is the smallest closed subset $\Omega\subset S^{n-1}$ such that 
$$
K=\left\{x\in\R^n:\langle x,u\rangle\leq h_K(u),\;\forall u\in\Omega\right\},
$$ 
and ${\rm supp}\,S_K$ is not contained in any closed hemisphere (see Lemma~2.5.6 in B\"or\"oczky, Figalli, Ramos \cite{BFR26}). 
If $K\in \mathcal{K}^n_{(o)}$, then the following Alexandrov-type variational formula is proved as Theorem~3.1 in Cai, Leng, Wu, Xi \cite{CLWX26}.

\begin{theorem}[Cai, Leng, Wu, Xi]
\label{Variation-Phimq}
Let $K\in \mathcal{K}^n_{(o)}$, $m\in\{1,\ldots,n-1\}$, $q\in\R\backslash\{o\}$ and $\Omega\subset\R^n$ be a closed subset that is not contained in a closed hemisphere. If $t\in(-t_0,t_0)$, then the "logarithmic family" of Wulff-shapes
\begin{align}
\label{Variation-Phimq-K}
K_t=&\left\{x\in\R^n:\langle x,u\rangle\leq h_t(u),\;\forall u\in\Omega\right\},\\
\label{Variation-Phimq-h}
h_t(u)=&h_0(u)\exp\left(t\varphi(u)+r(u,t)\right)  \mbox{ \ for $u\in \Omega$}
\end{align}
where $h_0:\Omega\to(0,\infty)$ and $\varphi:\Omega\to\R$ are continuous, $r(u,t)$ is Borel with $\lim_{t\to 0}\frac{r(u,t)}t=0$ uniformly in $u\in\Omega$, and $\inf_{u\in\Omega,\;t\in(-t_0,t_0) }h_t(u)>0$, satisfy that
\begin{align}
\label{Variation-Phimq-eq}
\lim_{t\to 0} \frac{\widetilde{\Psi}_{m,q}(K_t)-\widetilde{\Psi}_{m,q}(K)}t=&q\int_{\Omega}\varphi\,d\widetilde{A}_{m,q}(K_0,\cdot);\\
\label{Variation-Phim0-eq}
\lim_{t\to 0} \frac{\widetilde{\Psi}_{m,0}(K_t)-\widetilde{\Psi}_{m,0}(K)}t=&\int_{\Omega}\varphi\,d\widetilde{A}_{m,0}(K_0,\cdot).
\end{align}
\end{theorem}
\noindent{\bf Remarks.}
\begin{enumerate}
\item[(i)] We note that ${\rm supp}\,\widetilde{A}_{m,q}(K_0,\cdot)\subset {\rm supp}\,S_{K_0}\subset\Omega$ by \eqref{suppAmqKSK}.
\item[(ii)] If $K_t$ is only defined for $t\in[0,t_0)$, then we have $\lim_{t\to 0^+}$ in \eqref{Variation-Phimq-eq} and \eqref{Variation-Phim0-eq}.
\item[(iii)] Theorem~3.1 in  Cai, Leng, Wu, Xi \cite{CLWX26} actually states condition \eqref{Variation-Phimq-h} without the error term $r(u,t)$. However, their argument is based on the formula  
$$
\lim_{t\to 0} \frac{\varrho_{K_t}(u)-\varrho_{K_0}(u)}t=\varphi(\alpha_{K_0}(u))
$$
holding for $\HH^{n-1}$ a.e. $u\in \Theta_{K_0}$ (cf. \eqref{ThetaK-0}) according to 
Lemma~4.3 in Huang, Lutwak, Yang, Zhang \cite{HLYZ16}, which statement actually assumes exactly \eqref{Variation-Phimq-h}. 
We note that $\alpha_{K_0}(u)\in\Omega$ for any $u\in \Theta_{K_0}$ by \eqref{suppSK-regular-points} and (i).
\end{enumerate}

\begin{remark}
\label{Amqp-Lpsum}
 Let $p>0$, $q\in\R\backslash\{0\}$ and $m=1,\ldots,n-1$.
For $K,L\in \mathcal{K}^n_{(o)}$, there exists $R>1$  with $R^{-1}B^n\subset K,L\subset R B^n$, and hence for $t\geq 0$, \eqref{Lp-sum} says that
\begin{align*}
K+_p t\cdot L=&\left\{x\in\R^n:\langle x,u\rangle\leq h_t(u),\;\forall u\in S^{n-1}\right\}\\
h_t(u)=&\left(h_K(u)^p+t\cdot h_L(u)^p\right)^{\frac1p}=h_K(u)\left(1+t\cdot\frac{h_L(u)^p}{ph_K(u)^p}+O(t^2)\right)
\end{align*} 
for $u\in S^{n-1}$ as $t\to 0^+$ where the implied constant in $O(t^2)$ depends on $p$ and $R$ but not on $u$. Therefore, Theorem~\ref{Variation-Phimq} yields \eqref{Amqp} and \eqref{Am0p}.
\end{remark}

Finally, we turn to the weak convergence of $\widetilde{A}_{m,q}(K,\cdot)$. 
 The Hausdorff distance of two convex bodies $K,L\in \mathcal{K}^n_{o}$ is
\begin{align}
\label{deltaH-h}
\delta_H(K,L)=&\max\{|h_K(u)-h_L(u)|:u\in S^{n-1}\}\\
\label{deltaH-h}
=&\min\{r\geq 0:K\subset L+r B^n\mbox{ and }L\subset K+rB^n\}.
\end{align}
Then $\delta_H$ is a metric on $\mathcal{K}^n_{o}$, and we mean convergence of convex bodies in terms of this metric. According to the Blaschke Selection theorem, any bounded sequence of convex bodies contains a convergent subsequence. 

\begin{claim}
\label{normal-convergence}
If a sequence of $K_\ell\in \mathcal{K}^n_{(o)}$ tend to a convex body $K\in \mathcal{K}^n_{o}$, and for some $u\in \Theta_K$ (cf. \eqref{ThetaK-def}),  $w_\ell\in S^{n-1}$ is an exterior normal to $K_\ell$ at $x_\ell=\varrho_{K_\ell}(u)\cdot u$, then $\lim_{\ell\to\infty}w_\ell=\nu_K(x)$ for $x=\varrho_K(u)\cdot u\in\partial' K$.
\end{claim}
\begin{proof} We may assume that $\lim_{\ell\to\infty}w_\ell=w\in S^{n-1}$ and $\lim_{\ell\to\infty}x_\ell=y=tu$, $t>0$.  It follows from \eqref{deltaH-h} that
$$
\langle w,y\rangle=\lim_{\ell\to\infty}\langle w_\ell,x_\ell\rangle=\lim_{\ell\to\infty}h_{K_\ell}(w_\ell)=h_K(w);
$$
therefore, $y\in\partial K$ and $w$ is a normal vector at $y$. Since $u\in{\rm int}\,N_K(o)^*$, the open ray $(0,\infty)u$ intersects $\partial K$ only in $x$, and as $x\in\partial'K$, we have $w=\nu_K(x)$.
\end{proof}

Next we show that if $q>0$, then $\widetilde{\Psi}_{m,q}(L)$ is small provided that $L$ is a "almost flat".

\begin{claim}
\label{pancake}
Let $R>0$, $m\in\{1,\ldots,n-1\}$ and $q\in\R$. There exist $\gamma,\varepsilon_0>0$ depending on $n,m,q,R$ such that if $L\in \mathcal{K}^n_{(o)}$ is a convex body with 
$L\subset R\,B^n$, and $h_{L}(w), h_{L}(-w) \leq \varepsilon$ hold for some $w\in S^{n-1}$ and $\varepsilon\in(0,\varepsilon_0)$, then 
\begin{align}
\label{pancake-eq}
\widetilde{\Psi}_{m,q}(L)\leq& \gamma \varepsilon^{\frac{\min\{q,m\}}2}&&\mbox{if }q>0\\
\label{pancake-eq0}
\widetilde{\Psi}_{m,q}(L)\leq& \gamma\log\varepsilon&&\mbox{if }q=0\\
\label{pancake-eq-}
\widetilde{\Psi}_{m,q}(L)\geq& \gamma\varepsilon^{\frac{q}2}&&\mbox{if }q<0.
\end{align}
\end{claim}
\begin{proof}
For $\xi\in{\rm G}(n,m)$, let $v_\xi\in S^{n-1}\cap\xi$ satisfy that $\Pi_\xi w=\|\Pi_\xi w\|v_\xi$, and we write that $\xi\in\Gamma_1$ if $\|\Pi_\xi w\|\geq \sqrt{\varepsilon}$, and $\xi\in\Gamma_0$ if $\|\Pi_\xi w\|< \sqrt{\varepsilon}$.

For $\xi\in\Gamma_1$, we observe that $\|\Pi_\xi w\|=|\langle v,w\rangle|$, and since $w-\Pi_\xi w$ is orthogonal to $\xi$, any $x\in\xi\cap L$ can be written in the form $x=\lambda v_\xi+z$ where $\lambda\in\R$ and $z\in w^\bot\cap v_\xi^\bot$, and $|\langle x,w\rangle|\leq \varepsilon$ yields that $|\lambda|\cdot \|\Pi_\xi w\|\leq\varepsilon$.
We deduce that if $\xi\in\Gamma_1$, then
$$
h_{\xi\cap L}(v_\xi)\leq \frac{\varepsilon}{\|\Pi_\xi w\|}\leq \sqrt{\varepsilon}\mbox{ \ and \ }
h_{\xi\cap L}(-v_\xi)\leq \sqrt{\varepsilon},
$$
and hence $\xi\cap L \subset RB^n$ yields that
\begin{equation}
\label{pancake-Gamma1}
\HH^m(\xi\cap L)\leq 2R^{m-1}\omega_{m-1}\sqrt{\varepsilon}\mbox{ \ \ \ for }\xi\in\Gamma_1.
\end{equation}
On the other hand, it follows from \eqref{Gnm-u} that
$$
\nu_{n,m}(\Gamma_0)\leq \gamma_0 \varepsilon^{\frac{m}2}
$$
for a $\gamma_0>0$ depending on $n,m$.
Combining this estimate with \eqref{pancake-Gamma1} implies \eqref{pancake-eq}, \eqref{pancake-eq0} and \eqref{pancake-eq-}.
\end{proof}

If $q>0$ and $K\in \mathcal{K}^n_{o}$ satisfies ${\rm dim}\,K\leq n-1$, then \eqref{Psimqdef-sectionarea} and \eqref{xi-in-hyperplanes} yield that $\widetilde{\Psi}_{m,q}(K)=0$, and we define $\widetilde{A}_{m,q}(K,\cdot)\equiv 0$.

\begin{prop}
\label{Amq-weak-convergence}
Let $m\in\{1,\ldots,n-1\}$. If $q\in\R$, then $\widetilde{\Psi}_{m,q}(K)$ is continuous and $\widetilde{A}_{m,q}(K,\cdot)$ is weakly continuous for $K\in \mathcal{K}^n_{(o)}$, and if $q>0$, then $\widetilde{\Psi}_{m,q}(K)$ is continuous and $\widetilde{A}_{m,q}(K,\cdot)$ is weakly continuous for $K\in \mathcal{K}^n_{o}$.
\end{prop}
\begin{proof} In either cases, it is equivalent to prove that if $g:S^{n-1}\to\R$ is continuous, and a sequence $K_\ell\in \mathcal{K}^n_{(o)}$ tends to a suitable $K\in \mathcal{K}^n_{o}$, then 
\begin{align}
\label{Amq-weak-convergence-Psi}
\lim_{\ell\to\infty}\widetilde{\Psi}_{m,q}(K_\ell)=&\widetilde{\Psi}_{m,q}(K)\\
\label{Amq-weak-convergence-gA}
\lim_{\ell\to\infty}\int_{S^{n-1}}g\,d\widetilde{A}_{m,q}(K_\ell,\cdot)=&\int_{S^{n-1}}g\,d\widetilde{A}_{m,q}(K,\cdot).
\end{align}
We observe that if $K$ is a convex body, then for the Borel set
$\Upsilon=\Theta_K\cap\bigcap_\ell\Theta_{K_\ell}\subset U_K$,   $\HH^{n-1}$ a.e. point of $U_K$ is in $\Upsilon$ by \eqref{ThetaK-0}, and in turn, \eqref{U-Upsilon-section} yields that $\nu_{m,n}$ a.e. $\xi\in\mathcal{U}_K$ satisfies $\HH^{m-1}(\Upsilon\cap\xi)=\HH^{m-1}(U_K\cap\xi)$. Therefore, we can use $\Upsilon$ instead of $\Theta_K$ or  $\Theta_{K_\ell}\cap U_K$ in the integrals below. First, we consider the simpler case when $o\in{\rm int}\,K$.\\

\noindent{\bf Case 1.} $q\in\R$ and $o\in{\rm int}\,K$.

For any $t\in(0,1)$, we have 
\begin{equation}
\label{tK-Kell}
tK\subset K_\ell\subset\mbox{$\frac1t$}\,K  \mbox{ if $\ell$ is large;} 
\end{equation}
therefore, the monotonicity of $\widetilde{\Psi}_{m,q}(\cdot)$ yields \eqref{Amq-weak-convergence-Psi}.

It also follows from \eqref{tK-Kell} and Lemma~\ref{Amq-sn-1-lemma} (ii) that the sequence of integrals $\int_{S^{n-1}}g\,d\widetilde{A}_{m,q}(K_\ell,\cdot)$ is bounded. For any $\xi\in {\rm G}(n,m)$ and $u\in\Upsilon$, \eqref{tK-Kell} yields that 
$\lim_{\ell\to\infty}\HH^m(\xi\cap K_\ell)=\HH^m(\xi\cap K)$ and
$\lim_{\ell\to\infty}\varrho_{K_\ell}(u)=\varrho_K(u)$, and  Claim~\ref{normal-convergence} and the continuity of $g$ implies that $\lim_{\ell\to\infty}g\left(\alpha_{K_\ell}(u)\right)=g\left(\alpha_{K}(u)\right)$; therefore, we conclude \eqref{Amq-weak-convergence-gA} from
 Lemma~\ref{Amq-sn-1-lemma} (i) and the Lebesgue Dominated Convergence Theorem.\\

\noindent{\bf Case 2.} $q>0$ and $o\in\partial\,K$.

Let $K\subset R\,B^n$ for $R>0$, and hence we may assume that $K_\ell\subset 2R\,B^n$.

First, we assume that ${\rm dim}\,K\leq n-1$; in particular, $K\subset w^\bot\cap R\,B^n$ for some $w\in S^{n-1}$. Since $K_\ell\subset 2R\,B^n$, and $\lim_{\ell\to\infty}(h_{K_\ell}(w)+ h_{K_\ell}(-w))=0$, we deduce \eqref{Amq-weak-convergence-Psi} from Claim~\ref{pancake}, and, in turn, \eqref{Amq-weak-convergence-gA} using Lemma~\ref{Amq-sn-1-lemma} (ii).

Finally, we assume that $K$ is a convex body, and prove \eqref{Amq-weak-convergence-gA}.  Let $M>0$ satisfy that $|g(u)|\leq M$ for $u\in S^{n-1}$. As $q>0$, it follows that for any $\xi\in {\rm G}(n,m)$, we have
\begin{align}
\label{weak-cont-fat-xi}
\left|\HH^m\left(K_\ell\cap \xi\right)^{q-1}\mathcal{R}_m\left(\varrho_{K_\ell}^m\cdot g\circ\alpha_{K_\ell}\right)(\xi)\right|\leq&
M\int_{\xi\cap S^{n-1}}\HH^m\left(K_\ell\cap \xi\right)^{q-1}\varrho_{K_\ell}^m\,d\HH^{m-1}\\
\nonumber 
=&Mm\HH^m\left(K_\ell\cap \xi\right)^q
\leq Mm\omega_m^q(2R)^{qm}.
\end{align}

Let $\aleph\subset {\rm G}(n,m)$ be the Borel subset of all $\xi\in{\rm G}(n,m)\backslash\mathcal{T}_{m,K}$ such that
$\HH^{m-1}(\Upsilon\cap\xi)=\HH^{m-1}(U_K\cap\xi)$.
We deduce from  Lebesgue's Dominated Convergence Theorem, \eqref{weak-cont-fat-xi}, \eqref{U-Upsilon-section},
Lemma~\ref{numn-tangent-mspaces} and Lemma~\ref{Amq-sn-1-lemma} that 
\eqref{Amq-weak-convergence-gA} will follow  if for any $\xi\in\aleph$, we have
\begin{align}
\label{weak-cont-fat-xi-aleph}
\lim_{\ell\to\infty}\HH^m\left(K_\ell\cap \xi\right)^{q-1}\mathcal{R}_m\left(\varrho_{K_\ell}^m\cdot g\circ\alpha_{K_\ell}\right)(\xi)&=\\
\nonumber
\HH^m\left(K\cap \xi\right)^{q-1}\mathcal{R}_m\left(\varrho_K^m\cdot g\circ\alpha_K\right)(\xi)&
\end{align}
For $\xi\in\aleph$ with $\xi\cap U_K\neq\emptyset$, Claim~\ref{normal-convergence} and the continuity of $g$ yields that if $u\in\xi\cap \Upsilon$, then
$$
\lim_{\ell\to\infty}g\left(\alpha_{K_\ell}(u)\right)=g(\alpha_{K}(u)).
$$
Therefore, using again Lebesgue's Dominated Convergence Theorem for the integrals implied on $\xi\cap S^{n-1}$  for $\xi\in\aleph$ in \eqref{weak-cont-fat-xi-aleph}, the formula \eqref{Amq-weak-convergence-gA} will follow  if for any $u\in S^{n-1}\backslash \partial N_K(o)^*$, we have
\begin{align}
\label{weak-cont-fat-uUK}
\lim_{\ell\to\infty} \varrho_{K_\ell}(u)=&\varrho_K(u)&&\mbox{provided }u\in U_K,\\
\label{weak-cont-fat-uoutside}
\lim_{\ell\to\infty} \varrho_{K_\ell}(u)=&0=\varrho_K(u)&&\mbox{provided }u\in S^{n-1}\backslash N_K(o)^*.
\end{align}
Now, \eqref{weak-cont-fat-uUK} is a consequence of the fact that for any $k$ and $t\in(0,1)$, 
$$
t(C_{K,k}\cap K)\subset C_{K,k}\cap K_\ell\subset\mbox{$\frac1t$}\,(C_{K,k}\cap K)  \mbox{ if $\ell$ is large.}
$$
For \eqref{weak-cont-fat-uoutside}, we observe that if $u\in S^{n-1}\backslash N_K(o)^*$ and $t>0$, then $d>0$ holds for the distance $d$ of $tu$ from $K$, and hence $tu\not\in K_\ell$ if $\delta_H(K,K_\ell)<d$.
This proves \eqref{Amq-weak-convergence-gA}, that in turn readily yields \eqref{Amq-weak-convergence-Psi}.
\end{proof}

\section{Some properties of the $L_p$ centro-sectional measure}
\label{secLpcentrosectional}

For a convex body $K\in \mathcal{K}^n_{o}$, let
\begin{equation}
\label{Xi-def2}
\Xi_K=\{x\in\partial'K:\langle x,\nu_K(x)\rangle=0\}=\partial' K\cap\partial N_K(o)^*,
\end{equation}
and hence $tx\in\Xi_K$ holds for any $x\in\Xi_K$ and $t\in[0,1]$, and the continuity of $\nu_K$ on $\partial'K$ and $\HH^{n-1}(\partial K\backslash \partial'K)=0$ yield 
\begin{equation}
\label{Xi-closure}
\partial'K\cap {\rm cl}\,\Xi_K=\partial'K\cap \Xi_K \mbox{ and }\HH^{n-1}({\rm cl}\,\Xi_K)=\HH^{n-1}(\Xi_K).
\end{equation}
Readily, $\Xi_K=\emptyset$ if $K\in \mathcal{K}^n_{(o)}$.

For $m\in\{1,\ldots,n-1\}$, $p,q\in\R$ and convex body $K\in \mathcal{K}^n_{o}$ with $\HH^{n-1}(\Xi_K)=0$, we define the Borel $L_p$ centro-sectional measure $\widetilde{A}_{m,q,p}(K,\cdot)$ on $S^{n-1}$ by the formula  (cf. \eqref{Amq-hK0} for well-definedness)
\begin{equation}
\label{Amqp-hAmq}
d\widetilde{A}_{m,q,p}(K,\cdot)=h_K^{-p}\,d\widetilde{A}_{m,q}(K,\cdot),
\end{equation}
that satisfies 
$\widetilde{A}_{m,q,p}(K,\{h_K=0\})=0$ by \eqref{Amq-hK0}, and
 if $\lambda>0$, then
\begin{equation}
\label{Amqp-homogeneity}
\widetilde{A}_{m,q,p}(\lambda K,\cdot)=\lambda^{mq-p}\,d\widetilde{A}_{m,q,p}(K,\cdot).
\end{equation}
We deduce from \eqref{tildeAmq-bdK-eq} that if 
 $g:S^{n-1}\to[0,\infty)$ is a Borel measurable function, then
\begin{align}
\label{tildeAmqp-bdK-eq}
\int_{S^{n-1}}g\,d\widetilde{A}_{m,q,p}(K,\cdot)=
\int_{\partial' K} &g(\nu_K(x))\,\langle \nu_K(x),x\rangle^{1-p}\|x\|^{m-n}\\
\nonumber 
&\mathcal{R}^*_m\HH^m\left(K\cap \cdot\right)^{q-1}\left(\frac{x}{\|x\|}\right)\,d\HH^{n-1}(x). 
\end{align}
According to Remark~\ref{h-zero}, it is possible that the solution $h$ of the Monge-Amp\`ere equation \eqref{Amqp-Monge-Ampere-plarge} for $\widetilde{A}_{m,q,p}(K, \cdot)$ and $p>1$ zero at some $u\in S^{n-1}$ even if 
$d\widetilde{A}_{m,q,p}(K, \cdot)=f\,d\HH^{n-1}$ for a positive and continuous $f$.

For a convex set $X\subset w^\bot$ with ${\rm dim}\,X=n-1$, we write ${\rm relint}\,X$ to denote the relative interior of $X$ with respect to the topology of $w^\bot$.

\begin{lemma}
\label{Amqp-support}
If $m\in\{1,\ldots,n-1\}$, $p,q\in\R$ and $K\in \mathcal{K}^n_{o}$ is a convex body with $\HH^{n-1}(\Xi_K)=0$, then ${\rm supp}\,\widetilde{A}_{m,q,p}(K,\cdot)$ is not contained in any closed hemisphere.
\end{lemma}
\begin{proof}
We prove Lemma~\ref{Amqp-support} indirectly; therefore, we suppose that there exists $w\in S^{n-1}$ such that 
\begin{equation}
\label{Amqp-support-zero}
\widetilde{A}_{m,q,p}\left(K,\Theta\right)=0
\end{equation}
holds for the open hemisphere $\Theta$ centered at $w$.
There exist concave functions $\varphi,\psi:\Pi_{w^\bot}K\to\R$ such that $K=\{y+tw:y\in \Pi_{w^\bot}K\;\&\;-\psi(y)\leq t\leq \varphi(y)\}$, and we write $X_+=\{y+\varphi(y)w:y\in {\rm relint}\, \Pi_{w^\bot}K\}$. We observe that if $x\in X_+\cap\partial'K$, then $\langle \nu_K(x),w\rangle>0$. We deduce from \eqref{tildeAmqp-bdK-eq} and \eqref{Amqp-support-zero} that
\begin{equation}
\label{Amqp-support-intX+}
0=\int_{X_+\cap \partial' K} \langle \nu_K(x),x\rangle^{1-p}\|x\|^{m-n}\mathcal{R}^*_m\HH^m\left(K\cap \cdot\right)^{q-1}\left(\frac{x}{\|x\|}\right)\,d\HH^{n-1}(x).
\end{equation}
However, $X_+^0=X_+\backslash \partial N_K(o)^*$ is relatively open on $\partial K$, and the integrand on the right hand side of \eqref{Amqp-support-intX+} is positive at each $x\in  X_+^0$. Therefore, \eqref{Amqp-support-intX+} yields that $X_+\cap \partial'K\subset \Xi_K$ (cf. \eqref{Xi-def2}), which leads to a contradiction because
$\HH^{n-1} (X_+)\geq \HH^{n-1} (\Pi_{w^\bot}K)>0$ holds as orthogonal projection decreases Hausdorff measure.
\end{proof}

Now, we prepare for the statement and proof of our  core result Proposition~\ref{extremal-problem}.
For a convex body $K\in \mathcal{K}^n_{o}$, let $r(K)$ be the largest radius of any Euclidean ball contained in $K$. Since there exists $x_0\in K$ and origin symmetric ellipsoid $E$ such that $x_0+E\subset K\subset x_0+nE$ (for example, the John ellipsoid, see \cite{Sch14} and \cite{BFR26}), we deduce the existence of a $w\in S^{n-1}$ such that
\begin{equation}
\label{Steinhagen}
|\langle x,w\rangle|\leq 2nr(K)\mbox{ \ for }x\in K.
\end{equation}

Next, the Taylor formula with a reminder yields the following.

\begin{claim}
\label{Bernoulli}
For $\theta>0$ and $a_1>a_0\geq 0$, there exists $\gamma_0,\gamma_1>0$ depending on $\theta$, $a_0$, $a_1$ such that for $a\in[a_0,a_1]$, we have
\begin{align*}
(a+s)^\theta-a^\theta\geq & \gamma_0 s &&\mbox{if $a_0>0$ and $0\leq s< a_0$},\\
(a+s)^\theta-a^\theta\geq & \gamma_0 \min\{s,s^\theta\}&&\mbox{if $a_0=0$ and $0\leq s< a_1$},\\
a^\theta-(a-s)^\theta\leq & \gamma_1 s&&\mbox{if $a_0>0$ and $0\leq s< \frac{a_0}2$}.
\end{align*}
\end{claim}

We observe that if $\mu$ is a finite Borel measure on $S^{n-1}$ that is not concentrated on any closed hemisphere, then 
there exists $\delta\in(0,1)$ such that
\begin{equation}
\label{mu-delta}
\mbox{$\mu(S^{n-1})\leq 1/\delta$, and $\mu(\Sigma(w,\delta))\geq \delta$ for any $w\in S^{n-1}$.}
\end{equation}

\begin{prop}
\label{extremal-problem}
For $m\in\{1,\ldots,n-1\}$, $p>0$, $q>0$ and $\delta\in(0,1)$, there exist $0<r_\delta<R_\delta$ and $0<\lambda_\delta<\Lambda_\delta$ depending on $m,p,q,\delta$ with the following properties. If $\mu$ is a Borel measure  on $S^{n-1}$ such that $\mu(S^{n-1})\leq 1/\delta$, and $\mu(\Sigma(w,\delta))\geq \delta$ for any $w\in S^{n-1}$, and 
$$
\mathcal{F}_\mu=\left\{K\in \mathcal{K}^n_{o}\mbox{ convex body}:\int_{S^{n-1}}h_K^p\,d\mu\leq 1\right\},
$$
then there exists a $\widetilde{K}_\mu\in \mathcal{F}_\mu$ maximizing $\widetilde{\Psi}_{m,q}(K)$ among $K\in \mathcal{F}_\mu$, and there exists $x_\mu\in  \widetilde{K}_\mu$ such that
\begin{equation}
\label{extremal-problem-radii}
x_\mu+r_\delta B^n\subset \widetilde{K}_\mu\subset R_\delta B^n.
\end{equation}
\begin{enumerate}
\item[(i)] Assuming that $o\in{\rm int}\,\widetilde{K}_\mu$, we have
\begin{equation}
\label{extremal-problem-Amqp-mu}
\widetilde{A}_{m,q,p}(\widetilde{K}_\mu,\cdot)=m\widetilde{\Psi}_{m,q}\left(\widetilde{K}_\mu\right)\cdot \mu,
\end{equation}
and if $p\neq mq$ (cf. \eqref{Amqp-homogeneity}), then there exists $\lambda\in(\lambda_\delta,\Lambda_\delta)$ with 
\begin{equation}
\label{extremal-problem-Amqp-mu-lambda}
\widetilde{A}_{m,q,p}(\lambda\widetilde{K}_\mu,\cdot)=\mu.
\end{equation}
\item[(ii)] Assuming that $p>1$ and $\mu$ is discrete, we have $o\in{\rm int}\,\widetilde{K}_\mu$.
\end{enumerate}
\end{prop}
\begin{proof} We observe that if $K\in \mathcal{F}_\mu$ and $z\in K\backslash \{o\}$ and $v=z/\|z\|$, then $h_K(u)\geq \langle u,z\rangle\geq \|z\|\delta$ holds for any $u\in \Sigma(v,\delta)$, thus  
$$
1\geq \int_{\Sigma(v,\delta)}h_K^p\,d\mu\geq \|z\|^p\delta^{p+1},
$$ 
which in turn yields that $K\subset R_\delta B^n$ for $R_\delta=\delta^{-\frac{p+1}p}$.  Since $K\subset R_\delta B^n$ holds for $K\in \mathcal{F}_\mu$, the Blaschke Selection Theorem and the continuity of $\widetilde{\Psi}_{m,q}(K)$ (cf. Proposition~\ref{Amq-weak-convergence}) yield the existence of $\widetilde{K}_\mu\in \mathcal{F}_\mu$ maximizing $\widetilde{\Psi}_{m,q}(K)$. It follows by homogeneity that
\begin{equation}
\label{Minkowski-problem-Kmu}
\int_{S^{n-1}} h_{\widetilde{K}_\mu}^p\,d\mu=1.
\end{equation}
Since  $\delta^{\frac{1}p}B^n\in \mathcal{F}_\mu$, the existence of $r_\delta>0$ depending on $\delta,n,m,q,p$ with $r(\widetilde{K}_\mu)\geq r_\delta$ follows from Claim~\ref{pancake}, \eqref{Steinhagen} and the maximality of $\widetilde{\Psi}_{m,q}(\widetilde{K}_\mu)$.

Next, we assume that $o\in{\rm int}\,\widetilde{K}_\mu$, and hence \eqref{extremal-problem-Amqp-mu}
 is equivalent with the formula
\begin{equation}
\label{Minkowski-problem-claim0}
m\int_{S^{n-1}}\varphi h_K^p\,d\mu=\widetilde{\Psi}_{m,q}\left(\widetilde{K}_\mu\right)^{-1} \int_{S^{n-1}}\varphi\,d\widetilde{A}_{m,q}(\widetilde{K}_\mu,\cdot)
\end{equation}
for any continuous function $\varphi:\,S^{n-1}\to\R$.  
Assuming that $|t|$ is small, we consider the Wulff-shape
$$
K_t=\{x\in\R^n:\langle x,u\rangle\leq h_{\widetilde{K}_\mu}(u)\exp(t\varphi(u))\;\;\forall u\in S^{n-1}\},
$$
that satisfies $K_0=\widetilde{K}_\mu$ and $o\in{\rm int}\,K_t$; moreover,  the variational formula Theorem~\ref{Variation-Phimq} yields that
\begin{equation}
\label{Minkowski-problem-Psi-der}
\left.\frac{d }{dt}\,\widetilde{\Psi}_{m,q}(K_t)\right|_{t=0}=q\int_{S^{n-1}}\varphi\,d\widetilde{A}_{m,q}(\widetilde{K}_\mu,\cdot).
\end{equation}
As $h_{K_t}(u)\leq h_{\widetilde{K}_\mu}(u)\exp(t\varphi(u))$ for any $u\in S^{n-1}$, we have 
$$
\theta(t):=\int_{S^{n-1}}h_K^p\exp(pt\varphi)\,d\mu\geq \int_{S^{n-1}} h_{K_t}^p\,d\mu.
$$
It follows that $\theta(t)^{-\frac1p}K_t\in \mathcal{F}_\mu$, and hence the maximality of $\widetilde{\Psi}_{m,q}(\widetilde{K}_\mu)$ implies
$$
f(t):=\log \widetilde{\Psi}_{m,q}\left(\theta(t)^{-\frac1p}K_t\right)=
\log \widetilde{\Psi}_{m,q}\left(K_t\right)-\frac{mq}{p}\cdot\log\theta(t)\leq f(0).
$$
As $\theta(0)=1$ by \eqref{Minkowski-problem-Kmu} and $\theta'(0)=p\int_{S^{n-1}}\varphi h_K^p\,d\mu$, we deduce from \eqref{Minkowski-problem-Psi-der} that
$$
0=f'(0)=\widetilde{\Psi}_{m,q}\left(\widetilde{K}_\mu\right)^{-1}\cdot q\int_{S^{n-1}}\varphi\,d\widetilde{A}_{m,q}(\widetilde{K}_\mu,\cdot)-
mq\int_{S^{n-1}}\varphi h_K^p\,d\mu,
$$
proving \eqref{Minkowski-problem-claim0}, and in turn \eqref{extremal-problem-Amqp-mu}.

Finally, \eqref{extremal-problem-Amqp-mu-lambda} follows from  \eqref{Amqp-homogeneity} with $\lambda=m^{\frac{-1}{mq-p}}\widetilde{\Psi}_{m,q}\left(\widetilde{K}_\mu\right)^{\frac{-1}{mq-p}}$ if $p\neq mq$. For the existence of $\lambda_\mu$ and $\Lambda_\mu$ depending on $\delta,n,m,q,p$, we need suitable bounds on $\widetilde{\Psi}_{m,q}$. On the one hand,  \eqref{extremal-problem-radii} 
yields that $\widetilde{\Psi}_{m,q}\left(\widetilde{K}_\mu\right)\leq R_\delta^{mq}\widetilde{\Psi}_{m,q}(B^n)$. On the other hand, if 
a $\xi\in {\rm G}(n,m)$ intersects $x_\mu+\frac{r_\delta}2\, B^n$, then $\HH^m(\xi\cap \widetilde{K}_\mu)\geq \frac{r_\delta^m}{2^m}\,\omega_m$, and hence $\widetilde{\Psi}_{m,q}\left(\widetilde{K}_\mu\right)\geq \gamma$ holds for a $\gamma>0$ depending on $\delta,n,m,q,p$ by Claim~\ref{mplanes-ball}, completing the proof of (i).

For (ii), we assume that $p>1$ and ${\rm supp}\,\mu=\{u_1,\ldots,u_k\}$ where $k\geq n+1$, and hence the maximality of $\widetilde{\Psi}_{m,q}(\widetilde{K}_\mu)$ and the strict monotonicity of $\widetilde{\Psi}_{m,q}(\cdot)$ yield that  $\widetilde{K}_\mu$ is a polytope whose facet normals are among $u_1,\ldots,u_k$. 
We suppose that $o\in\partial\widetilde{K}_\mu$, and seek a contradiction. 

We may assume that $h_{\widetilde{K}_\mu}(u_i)=0$ if and only if $i=1,\ldots,\ell$ for some $1\leq \ell\leq k-1$, and \begin{equation}
\label{dimF1-n-1}
{\rm dim}\, F_1=n-1
\mbox{ \ for }F_1=F(\widetilde{K}_\mu,u_1).
\end{equation}
Let $\tilde{r}>0$ be such that $h_{\widetilde{K}_\mu}(u_i)>\tilde{r}$ for $i=\ell+1,\ldots,k$. 
We may fix a $w\in u_1^\bot\cap S^{n-1}$ and a $\theta\in (0,1)$ such that $\Sigma_0:=\tilde{r}\cdot\Sigma(w,\theta)\cap u_1^\bot\subset {\rm relint}\,F_1$, and then a $\tau>0$ such that
\begin{equation}
\label{tau-in-Kmu}
\{x\in\Sigma_0+\tau\,B^n:\langle x,u_1\rangle\leq 0\}\subset \widetilde{K}_\mu.
\end{equation}
Let $\mu(\{u_i\})=\alpha_i>0$ for $i=1,\ldots,k$. 
For small $t\geq 0$ and $\alpha=\frac{\alpha_1+\ldots+\alpha_\ell}{\alpha_{\ell+1}+\ldots+\alpha_k}$, we consider the convex body $L_t\in\mathcal{F}_\mu$ that is the family of all $x\in\R^n$ such that
\begin{align*}
\langle x,u_i\rangle\leq &t\mbox{ \ if }i=1,\ldots,\ell,\\
\langle x,u_i\rangle\leq &\left(h_{\widetilde{K}_\mu}(u_i)-\alpha t^p\right)^{\frac1p}\mbox{ \ if }i=\ell+1,\ldots,k,
\end{align*}
and hence $L_0=\widetilde{K}_\mu$.
During the argument, we write $\gamma_1,\gamma_2,\ldots>0$ to denote positive constants that are independent of the small parameter $t\geq 0$, and hence depend only on $\widetilde{K}_\mu$, $\delta$, $n$, $m$, $p$, $q$, $\tilde{r}$, $\tau$, $w$, $\theta$.

If $t>0$ is small, then 
\begin{equation}
\label{Sigma0-out-Kmu}
\Sigma_0+{\rm conv}\{o,tu_1\}\subset L_t\backslash {\rm int}\,\widetilde{K}_\mu,
\end{equation}
and we claim that
\begin{align} 
\label{rho-outside-pos}
\varrho_{L_t}(u)>&0=\varrho_{\widetilde{K}_\mu}(u)\mbox{ \ for }u\in S^{n-1}\backslash N_{\widetilde{K}_\mu}(o)^*,\\
\label{rho-inside-est}
\varrho_{L_t}(u)\geq &\varrho_{\widetilde{K}_\mu}(u)-\gamma_1t^p\mbox{ \ and $\tilde{r}<\varrho_{\widetilde{K}_\mu}(u)\leq R_\delta$ for }u\in S^{n-1}\cap N_{\widetilde{K}_\mu}(o)^*
\end{align}
where \eqref{rho-outside-pos} follows by definition.
For \eqref{rho-inside-est} and $u\in S^{n-1}\backslash N_{\widetilde{K}_\mu}(o)^*$, let $x=\varrho_{\widetilde{K}_\mu}(u)u$ and $x_t=\varrho_{L_t}(u)u$, and hence $x\in F(\widetilde{K}_\mu,u_i)$
and $x_t\in F(L_t,u_j)$ 
with $h_{L_t}(u_j)=\left(h_{\widetilde{K}_\mu}(u_j)-\alpha t^p\right)^{\frac1p}$
 for some $i,j\in\{\ell+1,\ldots,k\}$. We deduce that  $\varrho_{\widetilde{K}_\mu}(u)\geq \langle x,u_i\rangle=h_{\widetilde{K}_\mu}(u_i)>\tilde{r}$, and readily $\varrho_{\widetilde{K}_\mu}(u)\leq R_\delta$. As $\varrho_{L_t}(u)\leq R_\delta$ and $h_{L_t}(u_j)>\tilde{r}$ for small $t$, we have
$\langle u,u_j\rangle=\frac{h_{L_t}(u_j)}{\varrho_{L_t}(u)}\geq \frac{\tilde{r}}{R_\delta}$, and hence  $\langle x,u_j\rangle\leq h_{\widetilde{K}_\mu}(u_j)$ and Claim~\ref{Bernoulli} yield that
$$
\varrho_{\widetilde{K}_\mu}(u)-\varrho_{L_t}(u)=\frac{\langle x,u_j\rangle}{\langle u,u_j\rangle}-\frac{\langle x_t,u_j\rangle}{\langle u,u_j\rangle}\leq
\frac{h_{\widetilde{K}_\mu}(u_j)-h_{L_t}(u_j)}{\langle u,u_j\rangle}\leq \gamma_1t^p.
$$
We deduce from \eqref{rho-outside-pos}, \eqref{rho-inside-est} and Claim~\ref{Bernoulli} that for any $\xi\in{\rm G}(n,m)$, we have
\begin{equation}
\label{sections-Lt-Kmu-general}
\HH^m(L_t\cap\xi)^q\geq\HH^m(\widetilde{K}_\mu\cap\xi)^q-\gamma_2t^p.
\end{equation}
Next, we claim that there exists  a Borel subset $U_m\subset {\rm G}(n,m)$ such that for $\xi\in U_m$, we have
\begin{align}
\label{section-outside-m1}
\HH^1\left(L_t\cap\xi\right)^q-\HH^1\left(\widetilde{K}_\mu\cap\xi\right)^q\geq \gamma_3 \mbox{ and }\nu_{n,m}(U_m)\geq \gamma_4t&\mbox{ if } m=1,\\
\label{section-outside-m2}
\HH^m\left(L_t\cap\xi\right)^q-\HH^m\left(\widetilde{K}_\mu\cap\xi\right)^q\geq \gamma_5t \mbox{ and }\nu_{n,m}(U_m)\geq \gamma_6&\mbox{ if } m\geq 2.
\end{align}
For \eqref{section-outside-m1} when $m=1$, let $\eta\in(0,\frac{\pi}2)$ satisfy that $\tan\eta=\frac{t}{\tilde{r}}$, and let
$$
U':=\left\{v\in S^{n-1}:0<\langle v,u_1\rangle<\sin\eta\mbox{ and } \frac{\Pi_{u_1^\bot}v}{\|\Pi_{u_1^\bot}v\|}\in u_1^\bot\cap \Sigma(w,\theta)\right\}.
$$
It follows from 
\eqref{Gnm-u} and $\HH^{n-1}(u_1^\bot\cap \Sigma(w,\theta))>\gamma_7$ that 
$\HH^{n-1}(U')>\gamma_8t$.
We define $U_1:=\{\R v:v\in U'\}$, and hence $\nu_{n,1}(U_1)>\gamma_9t$. On the other hand, if $v\in U'$, then $\varrho_{\widetilde{K}_\mu}(v)=0$ and $\varrho_{L_t}(v)\geq \tilde{r}$ by
\eqref{Sigma0-out-Kmu}, and $\varrho_{\widetilde{K}_\mu}(-v)-\varrho_{L_t}(-v)\leq \gamma_1t^p$ by
\eqref{rho-inside-est}.
Therefore, if $t>0$ is small, then 
$\HH^1(\xi\cap L_t)\geq \HH^1(\xi\cap \widetilde{K}_\mu)+\frac{\tilde{r}}2$ for $\xi\in U_1$ where $\HH^1(\xi\cap \widetilde{K}_\mu)\leq R_\delta$, and in turn
Claim~\ref{Bernoulli} yields \eqref{section-outside-m1}.

For \eqref{section-outside-m2} when $m\geq 2$, let $\alpha\in(0,\frac{\pi}2)$ satisfy that $\theta=\cos\alpha$, and let
$$
U_m:=\left\{\xi\in{\rm G}(n,m):\xi\cap u_1^\bot\cap\Sigma\left(w,\cos\frac{\alpha}2\right)\neq\emptyset\mbox{ and }\xi\not\subset u_1^\bot\right\},
$$
and hence $\nu_{n,m}(U_m)>\gamma_{10}$ by 
Claim~\ref{mplanes-flatcones}. On the other hand, if $\xi\in U_m$, then any $u\in \xi\cap u_1^\bot\cap\Sigma\left(w,\cos\frac{\alpha}2\right)$ satisfies
${\rm conv}\left\{o,u_1^\bot\cap\Sigma\left(u,\cos\frac{\alpha}2\right)\right\}\subset \xi\cap F_1$, and there exists some $x\in\xi\cap (L_t\backslash \widetilde{K}_\mu)$ whose distance from $\xi\cap u_1^\bot$ is at least $t$ by \eqref{Sigma0-out-Kmu}. We deduce that $\HH^m(\xi\cap (L_t\backslash \widetilde{K}_\mu)>\gamma_{11}t$. Since
$\frac{\tau^m}2\,\omega_m\leq\HH^m(\xi\cap  \widetilde{K}_\mu)\leq R_\delta^m\omega_m$ by \eqref{tau-in-Kmu} and 
$$
\HH^m(\xi\cap  \widetilde{K}_\mu)-\HH^m(\xi\cap  \widetilde{K}_\mu\cap L_t)\leq \gamma_{12}t^p
$$
by \eqref{rho-inside-est} and Claim~\ref{Bernoulli}, 
it follows that
$$
\HH^m(\xi\cap L_t)-
\HH^m(\xi\cap  \widetilde{K}_\mu)>\gamma_{11}t- \gamma_{12}t^p>\gamma_{13}t
$$
if $t>0$ is small enough. Therefore, Claim~\ref{Bernoulli} yields  \eqref{section-outside-m2}.

We conclude from \eqref{sections-Lt-Kmu-general}, \eqref{section-outside-m1} and \eqref{section-outside-m2} that
$$
\widetilde{\Psi}_{m,q}(L_t)
-\widetilde{\Psi}_{m,q}(\widetilde{K}_\mu)\geq \gamma_{14}t-\gamma_2 t^p>\gamma_{15}t
$$
if $t>0$ is small enough. This contradicts the maximality of $\widetilde{\Psi}_{m,q}(\widetilde{K}_\mu)$, and proves
Proposition~\ref{extremal-problem}.
 \end{proof}

The same proof yields the even version of Proposition~\ref{extremal-problem} (cf. \eqref{mu-delta}). The only difference in the argument is that if $\mu$ is even, then the elements  of $\mathcal{F}_\mu$ are $o$-symmetric, and \eqref{pancake-eq} is replaced by \eqref{pancake-eq0} if $q=0$, and by \eqref{pancake-eq-} if $q<0$.

\begin{prop}
\label{extremal-problem-even}
For $m\in\{1,\ldots,n-1\}$, $p>0$, $q\in\R$, if $\mu$ is a finite even Borel measure  on $S^{n-1}$ that is not concentrated on any great subsphere, and 
$$
\mathcal{F}_\mu=\left\{K\in \mathcal{K}^n_{o}\mbox{ $o$-symmetric convex body}:\int_{S^{n-1}}h_K^p\,d\mu\leq 1\right\},
$$
then there exists a $\widetilde{K}_\mu\in \mathcal{F}_\mu$ maximizing $\widetilde{\Psi}_{m,q}(K)$ if $q\geq 0$, or minimizing $\widetilde{\Psi}_{m,q}(K)$ if $q<0$ among $K\in \mathcal{F}_\mu$. This $\widetilde{K}_\mu$ satisfies
\begin{equation}
\label{extremal-problem-Amqp-mu-even}
\widetilde{A}_{m,q,p}(\widetilde{K}_\mu,\cdot)=m\widetilde{\Psi}_{m,q}\left(\widetilde{K}_\mu\right)\cdot \mu,
\end{equation}
and hence if $p\neq mq$ (cf. \eqref{Amqp-homogeneity}), then there exists $\lambda>0$ such that 
\begin{equation}
\label{extremal-problem-Amqp-mu-lambda-even}
\widetilde{A}_{m,q,p}(\lambda\widetilde{K}_\mu,\cdot)=\mu.
\end{equation}
\end{prop}

The weak continuity of the $L_p$ centro-sectional measure readily follows from Proposition~\ref{Amq-weak-convergence} in some cases.

\begin{prop}
\label{Amqp-weak-convergence}
Let $m\in\{1,\ldots,n-1\}$. If $p,q\in\R$, then  $\widetilde{A}_{m,q,p}(K,\cdot)$ is weakly continuous for $K\in \mathcal{K}^n_{(o)}$, and if $q>0$ and $p\leq 1$, then  $\widetilde{A}_{m,q,p}(K,\cdot)$ is weakly continuous for $K\in \mathcal{K}^n_{o}$.
\end{prop}

Weak continuity of the $L_p$ centro-sectional measure may not hold if $1<p<m$ and $o\in\partial K$. Still, the crucial properties $\HH^{n-1}(\Xi_K)=0$ and $\widetilde{A}_{m,q,p}(K,S^{n-1})<\infty$ do hold for $K\in \mathcal{K}^n_o$ if $p>1$ and $K$ is obtained as a suitable limit.

\begin{lemma}
\label{KmXiK}
If $m\in\{1,\ldots,n-1\}$, $p>1$, $q>0$ and $K_\ell\in \mathcal{K}^n_{(o)}$ for $\ell\in\N$ tend to 
$K\in \mathcal{K}^n_o$ with ${\rm int}K\neq \emptyset$ such that
$\widetilde{A}_{m,q,p}(K_\ell,S^{n-1})$ stays bounded, then $\HH^{n-1}(\Xi_K)=0$ and $\widetilde{A}_{m,q,p}(K,S^{n-1})<\infty$.
 \end{lemma} 
\begin{proof} We may assume that $o\in\partial K$, and let $M>0$ such that 
\begin{equation}
\label{KmXiK-Ampq-upper}
\widetilde{A}_{m,q,p}(K_\ell,S^{n-1})\leq M
\end{equation}
for any $\ell\in\N$. First, we suppose that $\widetilde{A}_{m,q,p}(K,S^{n-1})=\infty$, and seek a contradiction. 

As $\widetilde{A}_{m,q,p}(K,S^{n-1})=\infty$, there exists an $\eta>0$ such that
$$
\int_{S^{n-1}}\max\{\eta,h_K\}^{1-p}\,d\widetilde{A}_{m,q}(K,\cdot)>3M.
$$
Now, $\max\{\eta,h_{K_\ell}\}^{1-p}$ tends uniformly to $\max\{\eta,h_K\}^{1-p}$, and according to Proposition~\ref{Amq-weak-convergence}, $\widetilde{A}_{m,q}(K_\ell,\cdot)$ tends weakly to $\widetilde{A}_{m,q}(K,\cdot)$, thus
\begin{equation}
\label{KmXiK-larger2M}
\int_{S^{n-1}}\max\{\eta,h_{K_\ell}\}^{1-p}\,d\widetilde{A}_{m,q}(K_\ell,\cdot)>2M
\end{equation}
if $\ell$ is large. However, the left hand side of \eqref{KmXiK-larger2M} is at most 
$\widetilde{A}_{m,q,p}(K_\ell,S^{n-1})$, contradicting \eqref{KmXiK-Ampq-upper}, and proving $\widetilde{A}_{m,q,p}(K,S^{n-1})<\infty$.
  
Next, we verify $\HH^{n-1}(\Xi_K)=0$. 
Since $z+rB^n\subset {\rm int}\,K$ and $K\subset {\rm int}\,RB^n$ for some $0<r<R$ and $z\in K$, we we may assume that 
$$
z+rB^n\subset K_\ell\subset RB^n\mbox{ \ for }\ell\in\N.
$$
For any bounded $X\subset \R^n\backslash\{z\}$, we consider the set
$$
\sigma(X)=\{z+\lambda(x-z):\,x\in X\mbox{ and }\lambda>0\}.
$$
We observe that $\sigma(X)$ is open if $X\subset{\partial}K$ is relatively open,  
and $\sigma(X)\cup\{z\}$ is closed if $X$ is compact and $z\not\in {\rm conv}\,X$.

We will use the weak continuity of the $(n-1)$th curvature measure. In particular, 
according to Theorem 4.2.1 and Theorem 4.2.3 in Schneider \cite{Sch14}, if $\beta\subset\R^n$ is open, then
\begin{equation}
\label{curvatureweakcont}
\liminf_{m\to\infty}\HH^{n-1}(\beta\cap {\partial}\,K_\ell)\geq \HH^{n-1}(\beta\cap {\partial}\,K).
\end{equation}

Let us suppose indirectly that $\HH^{n-1}(\Xi_K)>0$, and seek a contradiction.
Choose  a compact set $\widetilde{\Xi}\subset ({\rm cl}\Xi_K)\backslash \{o\}$ such that 
\begin{align*}
\HH^{n-1}(\widetilde{\Xi})=&\tau>0 \mbox{ and }z\not\in{\rm conv}\widetilde{\Xi}.
\end{align*}
We have $\widetilde{\Xi}\cap\partial'K\subset \Xi_K$ by \eqref{Xi-closure}.
Since $o$ is not an element of the closed set $\sigma(\widetilde{\Xi})\cup\{z\}$, there exists some $\eta>0$ such that
\begin{equation}
\label{KmXiK-etacond} 
(2\eta B^n)\cap \sigma(\widetilde{\Xi}+\eta B^n)=\emptyset.
\end{equation}
If $m\geq 2$, then we deduce from Corollary~\ref{mplanes-ball-section} the existence of $\gamma>0$ depending on $n,m,q,r,R$, such that
\begin{equation}
\label{KmXiK-Rm*}
\mathcal{R}^*_m\HH^m\left(K_\ell\cap \cdot\right)^{q-1}(u)\geq \gamma
\end{equation}
holds for any $u\in S^{n-1}$ and $\ell\in\N$.
Since $p>1$, we may choose $\varepsilon>0$ such that
\begin{equation}
\label{epsilonetaomega}
\begin{array}{rcll}
(2\varepsilon)^{1-p}\cdot\frac{2}{n\omega_n}\cdot\min\{\eta^{q-n},R^{q-n}\}\cdot(\tau/2)&>&M&\mbox{ if }m=1,\\[1ex]
(2\varepsilon)^{1-p}R^{m-n}\cdot\gamma\cdot(\tau/2)&>&M&\mbox{ if }m\geq 2.
\end{array}
\end{equation}
For any $x\in  \widetilde{\Xi}\cap\partial' K$, there exists $r_x\in(0,\eta)$ such that
\begin{equation}
\label{Bxcond}
h_K(u)\leq \varepsilon \mbox{ \ if $u\in S^{n-1}$ is exterior normal at $y\in {\partial}K\cap (x+r_xB^n)$, }
\end{equation}
and we define $B_x={\rm int}(x+r_xB^n)$. Let
$$
\mathcal{U}=\bigcup_{x\in \widetilde{\Xi}\cap\partial' K}(B_x\cap {\partial}K),
$$
which is a relatively open subset of ${\partial}K$ satisfying
\begin{enumerate}
\item[(a)] $(2\eta B^n)\cap \sigma(\mathcal{U})=\emptyset$,
\item[(b)] $\HH^{n-1}(\mathcal{U})\geq\tau$,
\item[(c)] $h_K(u)\leq \varepsilon$ if $u\in S^{n-1}$ is exterior normal at $x\in {\rm cl}\,\mathcal{U}$.
\end{enumerate}
It follows that (applying \eqref{curvatureweakcont} in the case (b')) that
 there exists $\ell_0$ such that if $\ell\geq \ell_0$, then
\begin{enumerate}
\item[(a')] $\|x\|\geq \eta$ if $x\in \sigma(\mathcal{U})\cap{\partial}K_\ell$,
\item[(b')] $\HH^{n-1}(\sigma(\mathcal{U})\cap{\partial}K_\ell)\geq\tau/2$,
\item[(c')] $h_{K_\ell}(u)\leq 2\varepsilon$ if $u\in S^{n-1}$ is exterior normal at 
$x\in \sigma(\mathcal{U})\cap{\partial}K_\ell$.
\end{enumerate}

For any $x\in \sigma(\mathcal{U})\cap{\partial}K_\ell$, (a') and $K_\ell\subset RB^n$ yield that
\begin{align*}
\|x\|^{q-n}&\geq \min\{\eta^{q-n},R^{q-n}\} \mbox{ \ if }m=1,\\
 \|x\|^{m-n}&\geq R^{m-n} \mbox{ \ if }m\geq 2.   \end{align*}
It follows first by \eqref{tildeAmqp-bdK-eq}, then by \eqref{KmXiK-Rm*}, (b'), (c') and \eqref{epsilonetaomega}, that for large $\ell$, if
$m=1$, then \eqref{A1q-eta-def} yields
$$
M\geq \widetilde{A}_{1,q,p}(K_\ell,S^{n-1})\geq
\frac{2}{n\omega_n}\int_{\sigma(\mathcal{U})\cap\partial'K_\ell} \langle \nu_K(x),x\rangle^{1-p}\|x\|^{q-n}\,d\HH^{n-1}(x)
>M,
$$
and if $m\geq 2$, then
$$
\widetilde{A}_{m,q,p}(K_\ell,S^{n-1})\geq
\int_{\sigma(\mathcal{U})\cap\partial'K_\ell} \langle \nu_K(x),x\rangle^{1-p}\|x\|^{m-n}\cdot\gamma\,d\HH^{n-1}(x)
>M,
$$
contradicting \eqref{KmXiK-Ampq-upper}, and proving
Lemma~\ref{KmXiK}.
\end{proof}

As Proposition~\ref{extremal-problem} and
Proposition~\ref{Amqp-weak-convergence} show, it is crucial to our study to find some properties that force the origin into the interior. 

\begin{prop}
\label{largep-oint}
If $m\in\{1,\ldots,n-1\}$, $p,q\in\R$ and $K\in \mathcal{K}^n_{o}$ is a convex body with $\HH^{n-1}(\Xi_K)=0$ and $\widetilde{A}_{m,q,p}(K,S^{n-1})<\infty$, then 
$K\in\mathcal{K}^n_{(o)}$ holds  provided that
\begin{enumerate}
\item[(i)] either $m=1$ and $p\geq \max\{1,q\}$;
 \item[(ii)] or $m\geq 2$ and $p\geq m$. 
\end{enumerate}
\end{prop}
\begin{proof} We suppose that $o\in\partial K$, and seek a contradiction.

According to B\"or\"oczky, Figalli, Ramos \cite{BFR26} Lemma 1.2.11, there exists a $v\in S^{n-1}$ and $t>0$ such that $tv\in{\rm int}\,K$ and $-v\in N_K(o)$. It follows that there exists $\sigma>0$ such that $v^\bot\cap \sigma B^n\subset {\rm relint}\,\Pi_{v^\bot}K$, and there exists a convex function $\varphi:v^\bot\cap \sigma B^n\mapsto[0,\infty)$ with $\varphi(o)=0$ such that $z+\varphi(z)v\in\partial K$, and we write $Y\subset \partial K$ to denote the graph of $\varphi$.
Thus for some $C>0$ (depending on $\sigma$ and $K$), we have
\begin{equation}
\label{largep-oint-projlength}
\|z\|\leq \|z+\varphi(z)v\|\leq C\|z\| \mbox{ \ for }z\in v^\bot\cap \sigma B^n.
\end{equation}
We consider $Z=\Pi_{v^\bot}((Y\cap\partial'K)\backslash\Xi_K)$ that consists of $\HH^{n-1}$ a.e. point of $v^\bot\cap \sigma B^n$. We observe that if $x=z+\varphi(z)v\in Y$ for $z\in Z\backslash\{o\}$, then
\begin{equation}
\label{largep-oint-xnuKx}
0<\langle\nu_K(x),x\rangle\leq \|x\|.
\end{equation}

\noindent{\bf Case 1.} $m=1$ and $p\geq \max\{1,q\}$.

Now, $o\in\partial K$ yields that $\mathcal{R}^*_m\HH^m\left(K\cap \cdot\right)^{q-1}\left(\frac{x}{\|x\|}\right)=\|x\|^{q-1}$  for $x\in Y\backslash\{o\}$ (cf. \eqref{A1q-eta-def}). As $\Pi_{v^\bot}$ decreases $\HH^{n-1}$, we deduce from \eqref{tildeAmqp-bdK-eq}, \eqref{largep-oint-projlength} and \eqref{largep-oint-xnuKx} that
\begin{align*}
\widetilde{A}_{m,q,p}(K,S^{n-1})>&\int_Y\langle \nu_K(x),x\rangle^{1-p}\|x\|^{q-n}\,d\HH^{n-1}(x)\\
\geq & 
\int_Y\|x\|^{1-p+q-n}\,d\HH^{n-1}(x)\\
\geq& C^{1-p+q-n}\int_{v^\bot\cap \sigma B^n}\|z\|^{1-p+q-n}\,d\HH^{n-1}(z)\\
=&C^{1-p+q-n}(n-1)\kappa_{n-1}\int_0^\sigma t^{q-p-1}\,dt=\infty.
\end{align*}
This contradiction proves that $o\in{\rm int}\,K$ holds in  (i).\\

\noindent{\bf Case 2.} $m\geq 2$ and $p\geq m$.

There exist $R>0$, $x_0\in K$ and $r>0$ such that $x_0+r B^n\subset K\subset R B^n$, and hence
$$
\mathcal{R}^*_m\HH^m\left(K\cap \cdot\right)^{q-1}(u)>\gamma
$$
holds for a $\gamma>0$ depending on $n,m,q,r,R$ by Corollary~\ref{mplanes-ball-section}.
Combining this estimate with \eqref{tildeAmqp-bdK-eq}, \eqref{largep-oint-projlength} and \eqref{largep-oint-xnuKx} yields that
\begin{align*}
\widetilde{A}_{m,q,p}(K,S^{n-1})>&\gamma\int_Y\langle \nu_K(x),x\rangle^{1-p}\|x\|^{m-n}\,d\HH^{n-1}(x)\\
\geq & 
\gamma\int_Y\|x\|^{1-p+m-n}\,d\HH^{n-1}(x)\\
\geq& \gamma C^{1-p+m-n}\int_{v^\bot\cap \sigma B^n}\|z\|^{1-p+m-n}\,d\HH^{n-1}(z)\\
=&\gamma C^{1-p+m-n}(n-1)\kappa_{n-1}\int_0^\sigma t^{m-p-1}\,dt=\infty.
\end{align*}
This contradiction proves that $o\in{\rm int}\,K$ holds in (ii), as well.
\end{proof}

Let us compare our $L_p$ centro-sectional measure $\widetilde{A}_{m,q,p}(K, \cdot)$ of a convex body $K\in\mathcal{K}^n_{o}$ to 
some $L_p$ dual curvature measure $d\widetilde{C}_{\tau,p}(K, \cdot)=h_K^{1-p}d\widetilde{C}_{\tau}(K, \cdot)$ introduced by Lutwak, Yang, Zhang \cite{LYZ18}.
We deduce from \eqref{Amq-eta-def} and
Corollary~\ref{mplanes-ball-section}  that
for $n\geq 3$, $m\in\{2,\ldots,n-1\}$, $q\in\R$ and $0<r<R$, there exists $\gamma\in(0,1)$ depending on $n,m,q,r,R$, such that
if $p\in\R$ and $K\in\mathcal{K}^n_{o}$ is a convex body with 
$K\subset RB^n$ and $r(K)\geq r$, then 
\begin{equation}
\label{Amqp-Cmp}
\gamma \,\widetilde{C}_{m,p}(K,\eta)\leq \widetilde{A}_{m,q,p}(K,\eta)\leq \gamma^{-1}\widetilde{C}_{m,p}(K,\eta)
\end{equation}
holds for any Borel set $\eta\subset S^{n-1}$.
On the other hand, for $n\geq 2$, $m=1$ and $p\in\R$, \eqref{A1q-Cq-est} and \eqref{A1q-Cq-bd}
yield that if $K\in\mathcal{K}^n_{o}$ is a convex body, then
\begin{align}
\label{A1qp-Cqp-est}
\widetilde{A}_{1,q,p}(K, \cdot)\geq &\frac{2}{n\omega_n}\,\widetilde{C}_{q,p}(K, \cdot);&& \\
\label{A1qp-Cqp-bd}
\widetilde{A}_{1,q,p}(K, \cdot)=&\frac{2}{n\omega_n}\,\widetilde{C}_{q,p}(K, \cdot)&&\mbox{if $o\in{\rm bd}\,K$.}
\end{align}

We note that some versions of Lemma~\ref{KmXiK}
and
Proposition~\ref{largep-oint} would follow from results in B\"or\"oczky, Fodor \cite{BoF19}
using \eqref{Amqp-Cmp} and \eqref{A1qp-Cqp-bd}, but not the exact statements we need.
The next remark shows that  we have to consider the possibility that the origin $o$ is in the boundary even if the $L_p$ centro-sectional measure has a positive continuous density.

\begin{remark}
\label{h-zero}
Given \eqref{Amqp-Cmp} and \eqref{A1qp-Cqp-bd},
 Example~7.2 in B\"or\"oczky, Fodor \cite{BoF19} shows that if $m=1$ and $q>p>1$, or $m\geq 2$, $1<p<m$ and $q\in\R$, then 
a solution $h$ of the Monge-Amp\`ere equation \eqref{Amqp-Monge-Ampere-plarge} for $\widetilde{A}_{m,q,p}(K, \cdot)$ might be zero at some $u\in S^{n-1}$ even if the $f$ in  \eqref{Amqp-Monge-Ampere-plarge} satifying $d\widetilde{A}_{m,q,p}(K, \cdot)=f\,d\HH^{n-1}$ is positive and continuous. Actually, if $m=1$ and $q=n$, then this is the unique solution of \eqref{Amqp-Monge-Ampere-plarge} according to Hug, Lutwak, Yang, Zhang \cite{HLYZ05}.
\end{remark}

\section{Regularity and partial Uniqueness for the $L_p$ centro-sectional Minkowski problem, and the proof of Theorem~\ref{regularity-uniqueness}}
\label{secregularity-uniqueness}

In this section, we discuss  the uniqueness and the regularity of the solution of $L_p$ centro-sectional Minkowski problem. Concerning uniqueness, more results will be obtained by considering Brunn-Minkowski-type inequalities in Section~\ref{secBM}.

The following simple uniqueness result about discrete measure has wide ranging consequences when $q\geq 1$ and $p>mq$.

\begin{prop}
\label{discrete-unique}
Let $m\in\{1,\ldots,n-1\}$, $q\geq 1$ and $p>mq$. If $K,L\in \mathcal{K}_{(o)}$ are polytopes with $\widetilde{A}_{m,q,p}(K,\cdot)=\widetilde{A}_{m,q,p}(L,\cdot)$, then $K=L$.
\end{prop}
\begin{proof} Let $\{u_1,\ldots,u_k\}\subset S^{n-1}$ be the  common support of $\widetilde{A}_{m,q,p}(K,\cdot)=\widetilde{A}_{m,q,p}(L,\cdot)$. We suppose that $K\neq L$, and seek a contradiction. Possibly interchanging the role of $K$ and $L$, there exists a $u_j$, $j\in\{1,\ldots,k\}$,  such that 
$h_{K}(u_j)>h_{L}(u_j)$. Let $t>1$ be minimal such that $K\subset tL$, and hence there exists an  $i\in\{1,\ldots,k\}$ such that $h_{K}(u_i)=h_{tL}(u_i)$. We deduce from 
$K\subset tL$ that $F(K,u_i)\subset F(tL,u_i)$, thus $\alpha^*_K(\{u_i\})\subset \alpha^*_{tL}(\{u_i\})$, and $\varrho_K(u)\leq \varrho_{tL}(u)$ and
$\mathcal{R}^*_m\HH^m\left(K\cap \cdot\right)(u)\leq \mathcal{R}^*_m\HH^m\left(tL\cap \cdot\right)(u)$ holds for any $u\in S^{n-1}$. Since $q\geq 1$, we conclude from
$h_{K}(u_i)=h_{tL}(u_i)$ and \eqref{Polytope-Amq-eq} in
Example~\ref{example-Amq-polytope} that
\begin{align*}
\widetilde{A}_{m,q,p}(K,\{u_i\})\leq &\widetilde{A}_{m,q,p}(tL,\{u_i\})=t^{mq-p}\widetilde{A}_{m,q,p}(L,\{u_i\})\\
=&t^{mq-p}\widetilde{A}_{m,q,p}(K,\{u_i\}),
\end{align*}
that is absurd as $t>1$ and $p>mq$.
\end{proof}

We recall that for $m\in\{1,\ldots,n-1\}$, $p,q\in \R$ and positive $h\in C^{2,\alpha}(S^{n-1})$ and $f\in C^{0,\alpha}(S^{n-1})$,
the Monge-Amp\'ere equation \eqref{Amqp-Monge-Ampere} for $\widetilde{A}_{m,q,p}(K,\cdot)$ can be written in the form
\begin{equation}
\label{Monge-Ampere-F-J}
F(\nabla^2h,h)=J_{f,m,q,p}(\nabla h,h)   
\end{equation}
where
$$
F(\nabla^2h,h)=\det(\nabla^2h+hI_{n-1})
$$
and $J(\nabla h,h)=J_{f,m,q,p}(\nabla h,h)$ satisfies
\begin{align}
\nonumber
\mathcal{R}^*_m\circ\mathcal{R}_m\left(\frac{\left(\|\nabla h\|^2+h_2\right)^{\frac{m}2}}m\right)^{q-1} \left(\frac{\nabla h+h\cdot {\rm id}}{\left(\|\nabla h\|^2+h_2\right)^{\frac12}}\right) \cdot J(\nabla h,h)& = \\
\label{Amqp-Monge-AmpereJ}
f\cdot h^{p-1}\left(\|\nabla h\|^2+h_2\right)^{\frac{n-m}2}.&
\end{align} 
In particular, if $t>0$, then
\begin{align}
\label{Monge-Ampere-homogeneityF}
F(\nabla^2(th),th)=&t^{n-1}F(\nabla^2h,h),\\
\label{Monge-Ampere-homogeneityJ}
J(\nabla (th),th)=&t^{p+n-1-mq} J(\nabla h,h).   
\end{align}
At any $u\in S^{n-1}$, $F(M,s)$ and $J(z,s)$ are functions of $(n-1)\times(n-1)$  symmetric matrix $M$, $s\in \R$ and $z\in u^\bot$, and we write $DF$ and $DJ$ to denote the derivative of $F$ and $J$. We observe that $F(M,s)-J(z,s)$ and $DF(M_0,s_0)(M,s)-DJ(z_0,s_0)(z,s)$ are non-linear uniformly elliptic operators.

\begin{lemma}
\label{regularity-mpq}
If $p,q\in\R$, $m\in\{1,\ldots,n-1\}$ and the $f$ in the Monge-Amp\`ere equation \eqref{Monge-Ampere-F-J} is a positive $C^{0,\alpha}$ function on $S^{n-1}$, then any solution $h$ of  is locally $C^{2,\alpha}$ on $\{h>0\}$.
\end{lemma}
\begin{proof}
Let $K\in\mathcal{K}^n_o$ be the convex body such that $h=h_K|_{S^{n-1}}$ is a solution of \eqref{Monge-Ampere-F-J}, and let $h(u_0)>0$ for a $u_0\in S^{n-1}$. We choose a compact neighborhood $V\subset S^{n-1}$ of $u_0$ such that $h(u)>0$ for $u\in V$. 
It follows from \eqref{A1q-eta-def} if $m=1$, and from Corollary~\ref{mplanes-ball-section} if $m\geq 2$ that $u\mapsto J_{f,m,q,p}(\nabla h,h)(u)$ is positive and continuous on $V$. Now the argument can be completed as in Theorem~4.1 in Huang, Zhao \cite{HuZ18} or in Lemma~3.2 in B\"or\"oczky,  Chen, Liu,  Saroglou \cite{BCLS} based on Caffarelli's regularity theory (cf. Figalli \cite{Fig17}).
\end{proof}

\begin{prop} 
\label{uniqueness-p-mq}
For $m\in\{1,\ldots,n-1\}$, $q\in\R$,  $p>mq$ and positive function
$f\in C^{0,\alpha}(S^{n-1})$, if $h_1,h_2\in C^{2,\alpha}(S^{n-1})$ are the positive solutions of the Monge-Amp\`ere equation \eqref{Monge-Ampere-F-J}, then $h_1=h_2$.
\end{prop}

\begin{proof}
Our argument is by contradiction. Without loss of  generality, we may assume that $h_1>h_2$ somewhere on $S^{n-1}$,
and hence there exists a constant $t>1$  such that
$th_2\geq h_1$ on $S^{n-1}$, and $th_2(u_0)=h_1(u_0)$ at some $u_0\in S^{n-1}.$

As $p>mq$, combining the facts that $h_1,h_2\in C^{2,\alpha}(S^{n-1})$ are the solutions of \eqref{Monge-Ampere-F-J} and $t>1$ with the homogeneity of $F,J$ (cf. \eqref{Monge-Ampere-homogeneityF} and \eqref{Monge-Ampere-homogeneityJ}), we have
$$
F(\nabla^2h_1,h_1)=J(\nabla h_1,h_1),
$$
and
\begin{align*}
F(t\nabla^2h_2,th_2)&=t^{n-1}F(\nabla^2h_2,h_2)
=t^{n-1}J(\nabla h_2,h_2)\\
=&t^{mq-p}J(t\nabla h_2,th_2)
\leq J(t\nabla h_2,th_2).
\end{align*}
It follows that (with $I=I_{n-1}$)
\begin{align*}
0\geq& F(t\nabla^2h_2,th_2)-F(\nabla^2h_1,h_1)+J(\nabla h_1,h_1)-J(t\nabla h_2,th_2)\\
=&\int_0^1 \frac{dF(\varepsilon t \nabla^2h_2+(1-\varepsilon)\nabla^2h_1,\varepsilon th_2+(1-\varepsilon)h_1)}{d\varepsilon}d\varepsilon-\\ 
&-\int_0^1\frac{dJ(\varepsilon\nabla th_2+(1-\varepsilon)\nabla h_1,\varepsilon th_2+(1-\varepsilon) h_1))}{d\varepsilon}d\varepsilon\\
=&\int_0^1 D F(\varepsilon t \nabla^2 h_2+(1-\varepsilon)\nabla^2h_1,\varepsilon th_2+(1-\varepsilon)h_1)d\varepsilon(\nabla^2(th_2-h_1)+(th_2-h_1)I)\\ 
&-\int_0^1D J(\varepsilon\nabla th_2+(1-\varepsilon)\nabla h_1)d\varepsilon\left(\nabla(th_2-h_1),th_2-h_1\right).
\end{align*}
In particular, $th_2-h_1$ satisfies an elliptic inequality. By the strong maximum principle (see, e.g., \cite{GT83} Theorem 3.5), we have $th_2\equiv h_1$ on $S^{n-1}$, and hence both $h_2$ and $th_2$ solve
\eqref{Monge-Ampere-F-J}.
We deduce from the homogeneity of $F$ and $J$ that
\begin{align*}
t^{n-1}F(\nabla^2h_2,h_2)&=F(\nabla^2th_2,th_2)=J(t\nabla h_2,th_2)\\
&=t^{n-1-mq+p}J(\nabla h_2,h_2)
=t^{n-1-mq+p}F(\nabla^2h_2,h_2).
\end{align*}
Therefore,  
$t^{p-mq}=1$, contradicting $t>1$, and in turn proving Proposition~\ref{uniqueness-p-mq}
\end{proof}

Readily, combining Lemma~\ref{regularity-mpq} and Proposition~\ref{uniqueness-p-mq}
proves Theorem~\ref{regularity-uniqueness}.

\section{The proof of Theorem~\ref{main} and part of Theorem~\ref{symmetric}}
\label{secproofmain}

\begin{proof}[The proofs of Theorem~\ref{main} and the existence part of Theorem~\ref{symmetric}]
Let $m\in\{1,\ldots,n-1\}$, $p>0$ and $q\in\R$. Now, the existence statement in
Theorem~\ref{symmetric} directly follows from Lemma~\ref{Amqp-support} and Proposition~\ref{extremal-problem-even}. 

For Theorem~\ref{main}, we assume that $p>1$ and $q>0$. Lemma~\ref{Amqp-support} yields that for any convex body $K\in\mathcal{K}^n_{o}$ with $\HH^{n-1}(\Xi_K)=0$, 
$\widetilde{A}_{m,q,p}(\widetilde{K}_\mu,\cdot)$ is not concentrated on any closed hemisphere.

Therefore, let $\mu$ be a finite Borel measure on $S^{n-1}$ not concentrated on any closed hemisphere. According to \eqref{mu-delta}, there exists $\delta\in(0,\frac12)$ such that
$$
\mbox{$\mu(S^{n-1})\leq \frac1{2\delta}$, and $\mu(\Sigma(w,2\delta))\geq 2\delta$ for any $w\in S^{n-1}$.}
$$
For any $\ell>\frac{2}{\delta}$, let $\aleph_\ell\subset S^{n-1}$ be a finite set  such that for any $x\in S^{n-1}$, there exists a $u\in \aleph_\ell$ such that $\|x-u\|\leq \frac{\delta}2$. Considering spherical Dirichlet-Voronoi cells with respect to $\aleph_\ell$, we obtain a partition $\{\Phi_u:\,u\in\aleph_\ell\}$ of $S^{n-1}$ into pairwise disjoint Borel sets such that $u\in \Phi_u$ and ${\rm diam}\,\Phi_u\leq \delta$ for $u\in\aleph_\ell$. For $\ell>\frac{2}{\delta}$, we define the discrete measure $\mu_\ell$ on $S^{n-1}$ by ${\rm supp}\,\mu_\ell=\aleph_\ell$ and
$\mu_\ell(\{u\})=\mu(\Phi_u)$ for $u\in\aleph_\ell$. Therefore, $\mu_\ell$ tends weakly to $\mu$, 
and
$$
\mbox{$\mu(S^{n-1})\leq 1/\delta$, and $\mu(\Sigma(w,\delta))\geq \delta$ for any $w\in S^{n-1}$.}
$$

It follows from Proposition~\ref{extremal-problem} that for any $\ell>\frac{2}{\delta}$, there exists a polytope $K_\ell\in\mathcal{K}^n_{(o)}$ such that $\widetilde{A}_{m,q,p}(K_\ell,\cdot)=\mu_\ell$, 
$r(K_\ell)\geq \lambda_\delta r_\delta$ and $K_\ell\subset \Lambda_\delta R_\delta$ where the constants $\lambda_\delta,r_\delta,\Lambda_\delta, R_\delta>0$  depend only on $\delta,n,m,q,p$.

According the Blaschke Selection theorem, there exists a subsequence $\{K_{\ell'}\}$ of  $\{K_{\ell}\}$ that tends to a convex body $K\in\mathcal{K}^n_{o}$. Since $\mu_{\ell'}$ tends weakly to $\mu$, the sequence $\widetilde{A}_{m,q,p}(K_\ell,S^{n-1})=\mu_\ell(S^{n-1})$ is bounded, and hence Lemma~\ref{KmXiK} yields that $\HH^{n-1}(\Xi_K)=0$ and $\widetilde{A}_{m,q,p}(K,S^{n-1})<\infty$.

For any continuous function $f:S^{n-1}\to \R$, 
$fh_{K_{\ell'}}^p$ tends uniformly to $fh_K^p$, and hence
$\HH^{n-1}(\Xi_K)=0$, the weak continuity of $\widetilde{A}_{m,q}$ (cf. Proposition~\ref{Amq-weak-convergence}) and the weak convergence of $\mu_{\ell'}$ to $\mu$ imply that
\begin{align}
\nonumber
\int_{S^{n-1}}fh_{K}^p\,d\widetilde{A}_{m,q,p}(K,\cdot)=&
\int_{S^{n-1}}f\,d\widetilde{A}_{m,q}(K,\cdot)=\lim_{\ell'\to\infty}
\int_{S^{n-1}}f\,d\widetilde{A}_{m,q}(K_{\ell'},\cdot)\\
\nonumber
=&\lim_{\ell'\to\infty}
\int_{S^{n-1}}fh_{K_{\ell'}}^p\,d\widetilde{A}_{m,q,p}(K_{\ell'},\cdot)\\
\label{mu-ell-Amqp}
=&
\lim_{\ell'\to\infty}\int_{S^{n-1}}fh_{K_{\ell'}}^p\,d\mu_{\ell'}
=\int_{S^{n-1}}fh_K^p\,d\mu;
\end{align}
therefore, $d\widetilde{A}_{m,q}(K,\cdot)=h_K^p\,d\mu$, and the restrictions of $\widetilde{A}_{m,q}(K,\cdot)$ and $\mu$ to $\{h_K>0\}$ coincide.

We note that if $p=mq$, then
$$
d\widetilde{A}_{m,q}(K_\ell,\cdot)=m\widetilde{\Psi}_{m,q}(K_\ell)\cdot h_{K_\ell}^p\,d\mu_\ell
$$
holds by Proposition~\ref{extremal-problem}, and hence
$$
d\widetilde{A}_{m,q}(K,\cdot)=m\widetilde{\Psi}_{m,q}(K)\cdot h_K^p\,d\mu.
$$

Finally, if $m=1$ and $p\geq q$, or $m\geq 2$ and $p\geq m$, then Proposition~\ref{largep-oint} implies that $o\in{\rm int}\,K$, completing the proof of Theorem~\ref{main}.
\end{proof}

If $o\in \partial K$ for a solution of Theorem~\ref{main} for a finite Borel measure $\mu$, then the restriction of $\widetilde{A}_{m,q,p}(K,\cdot)$ to the set $\{h_K=0\}=N_K(o)\cap S^{n-1}$ is the zero measure, while the restriction of $\mu$ to this set is arbitrary.

\section{$L_p$ Brunn-Minkowski inequalities and uniqueness}
\label{secBM}

This section demonstrates how $L_p$ Brunn-Minkowski inequalities for the centro-sectional measure intertwine with uniqueness of the solution of the $L_p$ centro-sectional Minkowski problem. First we characterize various equivalent forms of the $L_p$ centro-sectional Brunn-Minkowski and Minkowski inequalities.

\begin{lemma}
\label{LpBMMequivalent}
For $m\in\{1,\ldots,n-1\}$, $p>0$ and $q\neq 0$,
the following statements are equivalent if they hold for all $K,L\in\mathcal{K}^n_{(o)}$.
\begin{enumerate}
\item[(i)] $L_p$ centro-sectional Brunn-Minkowski inequality:
\begin{equation}
\label{LpBMMequivalent-BM}
\widetilde{\Psi}_{m,q}\left(\alpha\cdot K+_p\beta \cdot  L\right)^{\frac{p}{mq}}\geq 
\alpha\,\widetilde{\Psi}_{m,q}(K)^{\frac{p}{mq}}+\beta\,\widetilde{\Psi}_{m,q}(L)^{\frac{p}{mq}}
\end{equation}
 for $\alpha,\beta>0$ with equality if and only if 
$K$ and $L$ are dilates;
\item[(ii)] The function $f_{K,L}(\lambda)= \widetilde{\Psi}_{m,q}\left((1-\lambda)\cdot K+_p\lambda\cdot  L\right)^{\frac{p}{mq}}$ on $[0,1]$ is concave,
and is linear if and only if 
$K$ and $L$ are  dilates;
\item[(iii)] $L_p$ centro-sectional Minkowski inequality:
\begin{equation}
\label{LpBMMequivalent-Minkowski}
\int_{S^{n-1}}\frac{h_L^p}{h_K^p}\,d\widetilde{A}_{m,q}(K,\cdot)\geq m\widetilde{\Psi}_{m,q}(K)^{1-\frac{p}{mq}}\widetilde{\Psi}_{m,q}(L)^{\frac{p}{mq}},
\end{equation}
with equality if and only if $K$ and $L$ are dilates;
\item[(iv)] Assuming that $\widetilde{\Psi}_{m,q}(K)=\widetilde{\Psi}_{m,q}(L)=1$ and $\lambda\in(0,1)$, we have 
\begin{align*}
\widetilde{\Psi}_{m,q}\left((1-\lambda)\cdot K+_p\lambda\cdot  L\right)\geq & 1  \mbox{ \ if }q>0, \\
\widetilde{\Psi}_{m,q}\left((1-\lambda)\cdot K+_p\lambda\cdot  L\right)\leq &1  \mbox{ \ if }q<0,
\end{align*}
with equality if and only if $K=L$;
\item[(v)] If $\lambda\in(0,1)$, then
\begin{align*}
\widetilde{\Psi}_{m,q}((1-\lambda)K+_p\lambda L)\geq &\widetilde{\Psi}_{m,q}(K)^{1-\lambda}\widetilde{\Psi}_{m,q}(L)^{\lambda} \mbox{ \ if }q>0, \\
\widetilde{\Psi}_{m,q}((1-\lambda)K+_p\lambda L)\leq &\widetilde{\Psi}_{m,q}(K)^{1-\lambda}\widetilde{\Psi}_{m,q}(L)^{\lambda} \mbox{ \ if }q<0,
\end{align*}
with equality if and only if $K=L$.
\end{enumerate}
If $q>0$, then (i),  (iv) and (v) are equivalent even among convex bodies $K,L\in\mathcal{K}^n_{o}$.
\end{lemma}
\noindent{\bf Remarks.} 
\begin{itemize}
\item If $q=0$, $\lambda\in(0,1)$, $p>0$ and $K,L\in\mathcal{K}^n_{(o)}$, then the argument for Proposition~\ref{LpBMMequivalent} (only using \eqref{Variation-Phim0-eq} in Theorem~\ref{Variation-Phimq} instead of \eqref{Variation-Phimq-eq})
 yields that the following two inequalities are equivalent:
\begin{align}
\label{Psimqp-BM-qzero-equi}
\widetilde{\Psi}_{m,0}((1-\lambda)\cdot K+_p\lambda\cdot  L)\geq &(1-\lambda)\widetilde{\Psi}_{m,0}(K)+\lambda\widetilde{\Psi}_{m,0}(L),\\
\label{Psimqp-Minkowski-qzero-equi}
\int_{S^{n-1}}\frac{h_L^p}{h_K^p}\,d\widetilde{A}_{m,0}(K,\cdot)\geq & p\left(\widetilde{\Psi}_{m,0}(L)-\widetilde{\Psi}_{m,0}(K)\right)+m.
\end{align}
The equivalent equality conditions for \eqref{Psimqp-BM-qzero-equi} and \eqref{Psimqp-Minkowski-qzero-equi} are that equality holds if and only if $K=L$.
\item Naturally, the equivalence of the statements above hold if $K$ and $L$ are asssumed to the $o$-symmetric.
\end{itemize}
\begin{proof} 
We observe that (i) and (iv) are equivalent by the homogeneity of $\widetilde{\Psi}_{m,q}(\cdot)$ (cf. \eqref{Amq-homogeneous}), even among convex bodies $K,L\in\mathcal{K}^n_{o}$ if $q>0$. Readily, (i) yields (v) and (v) yields (iv).

(i) $\Longrightarrow$ (ii): Let $M_\lambda=(1-\lambda)\cdot K+_p\lambda\cdot  L$ for $\lambda\in[0,1]$. According to (i), (ii) follows 
if for any $t,s,\alpha,\beta\in(0,1)$ with $\alpha+\beta=1$, we have 
\begin{equation}
 \label{Mt-psum-concave}
\alpha\cdot M_t+_p\beta\cdot M_s\subset M_{\alpha t+\beta s}.
\end{equation}
If $z\in\partial'M_{\alpha t+\beta s}$, $u=\nu_{M_{\alpha t+\beta s}}(z)$ and $N=\alpha\cdot M_t+_p\beta\cdot M_s$, then (here Lemma~7.5.1 in \cite{BFR26} about regular boundary points of a Wulff shape yields the first inequality if $p\in(0,1)$)
\begin{align*}
h_{M_{\alpha t+\beta s}}(u)^p&=(1-\alpha t-\beta s)h_{K}(u)^{p}+(\alpha t+\beta s)h_{L}(u)^{p}\\
= &\alpha\left((1-t)h_{K}(u)^{p}+th_{L}(u)^{p}\right)
+\beta\left((1-s)h_{K}(u)^{p}+sh_{L}(u)^{p}\right)\\
&\geq\alpha h_{M_t}(u)^p+\beta h_{M_s}(u)^p\geq h_N(u)^p.
\end{align*}
Thus, we deduce \eqref{Mt-psum-concave}, and in turn the concavity of $f_{K,L}$, as  
$C=\{x\in\R^n:\langle x,z\rangle\leq h_C(\nu_C(z))\;\forall z\in\partial'C\}$ holds for any convex body $C\subset\R^n$ according to Lemma~2.5.6 in \cite{BFR26}.
In particular, $f_{K,L}$ is  linear if and only if $f_{K,L}(\frac12)=\frac12\,f_{K,L}(0)+\frac12\,f_{K,L}(1)$,
and hence (ii) has the same equality conditions as (i).\\

\noindent (ii) $\Longrightarrow$ (iii): We may assume that $\widetilde{\Psi}_{m,q}(K)=\widetilde{\Psi}_{m,q}(L)=1$ and $R^{-1}B^n\subset K,L\subset RB^n$ for some $R>1$, and hence $R^{-2}\leq h_L/h_K\leq R^2$ and $f_{K,L}(\lambda)\geq m=f_{K,L}(0)=f_{K,L}(1)$ for
$\lambda\in(0,1)$.
As $(1+t)^{\frac1p}= 1+\frac1p\cdot t+O(t^2)$ provided $|t|\leq \frac12$ where the implied constant in $O(\cdot)$ depends on $p$,
we deduce that if $\lambda>0$ is small and $u\in S^{n-1}$, then
\begin{align*}
\left((1-\lambda)h_K(u)^{p}+\lambda\,h_L(u)^p\right)^{\frac1p}=&h_K(u) \left(1+\lambda\cdot \left( \frac{h_L^p}{h_K^p}-1\right)\right)^{\frac1p}\\
=&
h_K(u)\exp\left(\frac{\lambda}p\left( \frac{h_L^p}{h_K^p}-1\right)+O(\lambda^2)\right)
\end{align*}
where the implied constant in $O(\cdot)$ depends on $p,R$, but not on $u$.
It follows from (ii) and Theorem~\ref{Variation-Phimq} for the Wulff shape that
\begin{equation}
\label{logMequivfprime00}
0\leq f'_{K,L}(0)=\frac{1}m\int_{S^{n-1}}\left(\frac{h_L^p}{h_K^p}-1\right)\,d\widetilde{A}_{m,q}(K,\cdot),
\end{equation}
proving (iii). Equality in \eqref{logMequivfprime00} yields $f'_{K,C}(0)=0$, and hence $f_{K,L}(\lambda)$ is linear.\\

\noindent (iii) $\Longrightarrow$ (i): First we assume that $q>0$. Using  the notation $M_\lambda=(1-\lambda)\cdot K+_p\lambda\cdot L$ for $\lambda\in (0,1)$, we deduce from
Lemma~2.5.6 and Lemma~7.5.1 in \cite{BFR26} that 
$h_{M_\lambda}(u)=\left((1-\lambda)h_K(u)^{p}+\lambda\,h_L(u)^p\right)^{\frac1p}$ for $S_{M_\lambda}$
(and hence $\widetilde{A}_{m,q}(M_\lambda,\cdot)$) a.e. $u\in S^{n-1}$; therefore, (iii) implies
\begin{align}
\nonumber
m\widetilde{\Psi}_{m,q}(M_\lambda)=&\int_{S^{n-1}}\frac{(1-\lambda)h_K^p+\lambda\,h_L^p}{h_{M_\lambda}^p}\,d\widetilde{A}_{m,q}(M_\lambda,\cdot)\\
\nonumber
=&(1-\lambda)\int_{S^{n-1}}\frac{h_K^p}{h_{M_\lambda}^p}\,d\widetilde{A}_{m,q}(M_\lambda,\cdot)+
\lambda\int_{S^{n-1}} \frac{h_L^p}{h_{M_\lambda}^p}\,d\widetilde{A}_{m,q}(M_\lambda,\cdot)\\
\label{plus-or-minus}
\geq& m(1-\lambda)\widetilde{\Psi}_{m,q}(M_\lambda)^{1-\frac{p}{mq}}\widetilde{\Psi}_{m,q}(K)^{\frac{p}{mq}}+\\
\nonumber
&+m\lambda\widetilde{\Psi}_{m,q}(M_\lambda)^{1-\frac{p}{mq}}\widetilde{\Psi}_{m,q}(L)^{\frac{p}{mq}}.
\end{align}
We conclude that $\widetilde{\Psi}_{m,q}(M_\lambda)^{\frac{p}{mq}}\geq (1-\lambda)\widetilde{\Psi}_{m,q}(K)^{\frac{p}{mq}}+\lambda\,\widetilde{\Psi}_{m,q}(L)^{\frac{p}{mq}}$, which in turn yields (i) by (iv). In addition, having equality in (i) with $\alpha=1-\lambda$ and $\beta=\lambda$ and equality in (iii) are equivalent.

If $q<0$, then we have "$\leq$" in \eqref{plus-or-minus}, but the exponent $\frac{p}{mq}<0$, thus we obtain (i) again.
\end{proof}

For $q\geq 1$,  first we prove the $L_p$ centro-sectional Minkowski inequality in a special case using the variational method and the uniqueness result for discrete measures.

\begin{lemma}
\label{polytope-Minkowski}
For $m\in\{1,\ldots,n-1\}$, $q\geq 1$ and $p>mq$, if $K\in \mathcal{K}^n_{(o)}$ is a polytope and   $L\in \mathcal{K}^n_{o}$ is a convex body, then
$$
\int_{S^{n-1}}h_L^p\,d\widetilde{A}_{m,q,p}(K,\cdot)\geq m\widetilde{\Psi}_{m,q}(K)^{1-\frac{p}{mq}}\widetilde{\Psi}_{m,q}(L)^{\frac{p}{mq}},
$$
with equality if and only if $K$ and $L$ are dilates.
\end{lemma}
\begin{proof}
It is equivalent to prove that if
$L\in \mathcal{K}^n_{o}$ is any convex body with $\widetilde{\Psi}_{m,q}(L)=\widetilde{\Psi}_{m,q}(K)$, then
\begin{equation}
\label{polytope-Minkowski-eq0}
\int_{S^{n-1}}h_L^p\,d\widetilde{A}_{m,q,p}(K,\cdot)\geq m\widetilde{\Psi}_{m,q}(K),
\end{equation}
with equality if and only if $K=L$.

We consider the discrete measure $\mu=\widetilde{A}_{m,q,p}(K,\cdot)$, and the family 
$$
\widetilde{\mathcal{F}}=\left\{L\in \mathcal{K}^n_{o}\mbox{ convex body}:\widetilde{\Psi}_{m,q}(L)=\widetilde{\Psi}_{m,q}(K)\right\}.
$$
According to Proposition~\ref{extremal-problem} applied with our $\mu$, there exists a $\widetilde{K}\in \mathcal{K}^n_{(o)}\cap \widetilde{\mathcal{F}}$  minimizing $\int_{S^{n-1}}h_L^p\,d\mu$ for $L\in\widetilde{\mathcal{F}}$, and  hence
\begin{equation}
\label{L-widetilde-K-min}
\int_{S^{n-1}}h_L^p\,d\widetilde{A}_{m,q,p}(K,\cdot)\geq \int_{S^{n-1}}h_{\widetilde{K}}^p\,d\widetilde{A}_{m,q,p}(K,\cdot)
\end{equation}
for any $L\in\widetilde{\mathcal{F}}$. In addition, Proposition~\ref{extremal-problem} also says that there exists some $\lambda>0$ such that
$\widetilde{A}_{m,q,p}(\lambda\widetilde{K},\cdot)=\widetilde{A}_{m,q,p}(K,\cdot)$. We deduce from Proposition~\ref{discrete-unique} that $\lambda\widetilde{K}=K$, and hence $\lambda=1$ follows from  $\widetilde{\Psi}_{m,q}(\widetilde{K})=\widetilde{\Psi}_{m,q}(K)$. Therefore, \eqref{L-widetilde-K-min} yields \eqref{polytope-Minkowski-eq0}. 

If equality holds in \eqref{polytope-Minkowski-eq0} for some convex body $L_0\in \mathcal{K}^n_{o}$  with $\widetilde{\Psi}_{m,q}(L_0)=\widetilde{\Psi}_{m,q}(K)$,
then $L_0$ is also a minimizer of $\int_{S^{n-1}}h_L^p\,d\mu$ for $L\in\widetilde{\mathcal{F}}$, and hence the argument above implies that
$L_0=K$.
\end{proof}

We recall Jensen's inequality stating that if $a,b>0$, $\lambda\in(0,1)$ and $\theta>\tau$ with $\theta,\tau\neq 0$, then
\begin{equation}
\label{Jensen}
((1-\lambda)a^\theta+\lambda b^\theta)^{\frac1{\theta}}\geq ((1-\lambda)a^\tau+\lambda b^\tau)^{\frac1{\tau}},
\end{equation}
with equality if and only if $a=b$, and that if $p>0$, $\alpha,\beta>0$ and 
 $K,L\in \mathcal{K}^n_{o}$ are convex bodies, then the $L_p$ sum satisfies
\begin{align*}
\alpha\cdot K+_{p}\beta\cdot L=&\left\{x\in\R^n:\langle x,u\rangle\leq  \left(\alpha h_K(u)^p+\beta h_L(u)\right)^{\frac1p}\right\},\\
h_{\alpha\cdot K+_{p}\beta\cdot L}(u)=&\left(\alpha h_K(u)^p+\beta h_L(u)\right)^{\frac1p}\mbox{ \ if $p\geq 1$ and $u\in S^{n-1}$}.
\end{align*}
We observe that $\alpha\cdot K+_{p}\beta\cdot L\in \mathcal{K}^n_{o}$ is a convex body as for any $w\in S^{n-1}$ there exists a $u\in \Omega_w$ for the open hemisphere $\Omega_w$ centered at $w$ such that $h_K(u)>0$ or $h_L(u)>0$. We also observe that the $L_p$ sum is continuous in the sense that if convex bodies $K_\ell\in\mathcal{K}^n_{o}$ tend to $K$ and $L_\ell\in\mathcal{K}^n_{o}$ tend to $L$, then $\alpha\cdot K_\ell+_{p}\beta\cdot L_\ell$ tends to $\alpha\cdot K+_{p}\beta\cdot L$.

\begin{claim}
\label{Lp-Lprime}
For $0< p'<p$, $\lambda\in(0,1)$ and convex bodies $K,L\in \mathcal{K}^n_{o}$, we have
\begin{equation}
\label{Lp-Lprime-eq}
(1-\lambda)\cdot K+_{p'}\lambda\cdot L\subset (1-\lambda) K+_{p}\lambda\cdot L,
\end{equation}
and assuming $p'\geq 1$, equality holds if and only if $K=L$.
\end{claim}
\begin{proof}
The Jensen inequality \eqref{Jensen} yields directly \eqref{Lp-Lprime-eq}.

 Let us assume that $p'\geq 1$ and  equality holds in \eqref{Lp-Lprime-eq}. It follows from the equality conditions for the Jensen inequality \eqref{Jensen} that if $u\not\in N_K(o)\cup N_L(o)$ for $u\in S^{n-1}$, then $h_K(u)=h_L(u)$. 
 
 We suppose indirectly that $K\neq L$, and hence $N_K(o)\neq N_L(o)$. Therefore, we may assume the existence of a $u\in (S^{n-1}\cap \partial N_K(o))\backslash N_L(o)$. We consider a sequence  $u_\ell\in S^{n-1}\backslash(N_K(o)\cup N_L(o))$ tending to $u$. As $h_K$ is continuous, we have $\lim_{\ell\to\infty} h_K(u_\ell)=h_K(u)=0$. However, $h_L(u_\ell)=h_K(u_\ell)$, thus $h_L(u)=0$, as well. This contradiction proves $K=L$.
\end{proof}

If $q\geq 1$, then we are ready to prove our $L_p$ centro-sectional Minkowski and Brunn-Minkowski inequalities.

\begin{prop}
\label{Psimqp-qlarger1}
Let $m\in \left\{1,\ldots,n-1\right\}$, $q\geq 1$ and  $p> mq$.
If $K,L\in\mathcal{K}^n_{(o)}$, then
 \begin{equation}
\label{Psimqp-Minkowski-qlarger1-eq}
\int_{S^{n-1}}h_L^p\,d\widetilde{A}_{m,q,p}(K,\cdot)\geq  m\widetilde{\Psi}_{m,q}(K)^{1-\frac{p}{mq}}\widetilde{\Psi}_{m,q}(L)^{\frac{p}{mq}},
\end{equation}
and if $K,L\in\mathcal{K}^n_{o}$ and $\alpha,\beta>0$, then
\begin{equation}
\label{Psimqp-BM-qlarger1-eq}
\widetilde{\Psi}_{m,q}(\alpha\cdot K+_p\beta\cdot  L)^{\frac{p}{mq}}\geq \alpha\widetilde{\Psi}_{m,q}(K)^{\frac{p}{mq}}+\beta\widetilde{\Psi}_{m,q}(L)^{\frac{p}{mq}},
\end{equation}
and equality holds in \eqref{Psimqp-Minkowski-qlarger1-eq} or in \eqref{Psimqp-BM-qlarger1-eq} if and only if $K$ and $L$ are dilates.
\end{prop}
\begin{proof} Let $q\geq 1$. As $\widetilde{A}_{m,q}(K,\cdot)$ is weakly continuous (cf. Proposition~\ref{Amq-weak-convergence}), we deduce the  inequality \eqref{Psimqp-Minkowski-qlarger1-eq} for any 
$K,L\in\mathcal{K}^n_{(o)}$ and $p>mq$
from Lemma~\ref{polytope-Minkowski}.
In turn, \eqref{Psimqp-Minkowski-qlarger1-eq} yields \eqref{Psimqp-BM-qlarger1-eq} for any $K,L\in\mathcal{K}^n_{o}$ and $p>mq$ by Lemma~\ref{LpBMMequivalent} and the continuity of $\widetilde{\Psi}_{m,q}(\cdot)$.

According to Lemma~\ref{LpBMMequivalent}, the characterization of inequality for \eqref{Psimqp-Minkowski-qlarger1-eq} and \eqref{Psimqp-BM-qlarger1-eq} follows
if for any convex bodies $K,L\in\mathcal{K}^n_{o}$ satisfying $K\neq L$ and $\widetilde{\Psi}_{m,q}(K)=\widetilde{\Psi}_{m,q}(L)=1$ and $\lambda\in(0,1)$, we have 
\begin{equation}
\label{qlarge1-K-not-L}
\widetilde{\Psi}_{m,q}\left((1-\lambda)\cdot K+_p\lambda\cdot  L\right)> 1.
\end{equation}
Let $p'\in (mq,p)$. It follows from \eqref{Psimqp-BM-qlarger1-eq} that $\widetilde{\Psi}_{m,q}\left((1-\lambda)\cdot K+_{p'}\lambda\cdot  L\right)\geq 1$. Therefore, combining $K\neq L$, Claim~\ref{Lp-Lprime} and the strict monotonicity of $\widetilde{\Psi}_{m,q}(\cdot)$
implies \eqref{qlarge1-K-not-L}, and in turn proves 
Proposition~\ref{Psimqp-qlarger1}.
\end{proof}

For the rest of the section, we need the notion of a star body and radial sum. Given a continuous function $\varrho:S^{n-1}\to(0,\infty)$, the corresponding star body $S$ is
$$
S=\{tu:\,u\in S^{n-1}\mbox{ and }0\leq t\leq \varrho(u)\},
$$
and $\varrho=\varrho_S$ is the radial function of $S$ satisfying $\varrho_S(u)=\max\{t\geq 0:tu\in S\}$. In particular, any $K\in \mathcal{K}^n_{(o)}$ is a star body. For $p>0$, $\lambda\in(0,1)$ and star bodies $S,T\subset\R^n$, their radial $L_p$ combination $(1-\lambda)\cdot S\widetilde{+}_p \lambda\cdot T$ is a star body defined by Lutwak \cite{Lut93} by the formula
$$
\varrho_{(1-\lambda)\cdot S\widetilde{+}_p \lambda T}=\left[
(1-\lambda)\varrho_S^p+\lambda\varrho_T^p\right]^{\frac1p}.
$$
Lemma~3.1 in Xi, Zhang \cite{XiZ22} proves the following lemma about the relation between the $L_p$ radial sum and the $L_p$
Minkowski sum.

\begin{lemma}\cite{XiZ22}
\label{lp-include} 
If $p>0$, $\lambda\in (0,1)$ and $K,L\in \mathcal{K}^n_{(o)}$, then
$$
(1-\lambda)\cdot K\widetilde{+}_p\lambda\cdot L\subset (1-\lambda)\cdot K+_p \lambda\cdot L,
$$
with equality if and only if $K$ and $L$ are dilates.
 \end{lemma}

\begin{prop}
\label{Psimqp-q01}
Let $m\in \left\{1,\ldots,n-1\right\}$, $0<q\leq 1$ and  $p\geq m$.
If $K,L\in\mathcal{K}^n_{(o)}$, then
 \begin{equation}
 \label{Psimqp-Minkowski-q01-eq}
\int_{S^{n-1}}h_L^p\,d\widetilde{A}_{m,q,p}(K,\cdot)\geq  m\widetilde{\Psi}_{m,q}(K)^{1-\frac{p}{mq}}\widetilde{\Psi}_{m,q}(L)^{\frac{p}{mq}},
\end{equation}
with equality if and only if $K$ and $L$ are dilates,
and if $K,L\in\mathcal{K}^n_{o}$ and $\alpha,\beta>0$, then
\begin{equation}
\label{Psimqp-BM-q01-eq}
\widetilde{\Psi}_{m,q}(\alpha\cdot K+_p\beta\cdot  L)^{\frac{p}{mq}}\geq \alpha\widetilde{\Psi}_{m,q}(K)^{\frac{p}{mq}}+\beta\widetilde{\Psi}_{m,q}(L)^{\frac{p}{mq}},
\end{equation}
where assuming, in addition, that either $p>m$ or $K,L\in\mathcal{K}^n_{(o)}$,
equality holds if and only if $K$ and $L$ are dilates.
\end{prop}
\begin{proof}
If $p\geq m$, $q\in(0,1]$, $\lambda\in(0,1)$ and $K,L\in\mathcal{K}^n_{(o)}$ satisfy
$\widetilde{\Psi}_{m,q}(K)=\widetilde{\Psi}_{m,q}(L)=1$, then applying first
Lemma~\ref{lp-include}, secondly the Jensen inequality \eqref{Jensen} with $\theta=\frac{m}p$ and $\tau=1$, and after that with $\theta=1$ and $\tau=q$, we deduce that
\begin{align}
 \nonumber
  \widetilde{\Psi}_{m,q}((1-\lambda)\cdot K+_p\lambda\cdot L)&=\frac{1}{m^q}\int_{{\rm G}(n,m)}R_m(\rho_{(1-\lambda)\cdot K+_p\lambda\cdot L}^m)^q\,d\nu_{n,m}\\
  \label{lp-include-qpos}
 &\geq\frac{1}{m^q}\int_{{\rm G}(n,m)}R_m(\rho_{(1-\lambda)\cdot K\widetilde{+}_p\lambda\cdot L}^m)^q\,d\nu_{n,m}\\
 \nonumber
 &=\frac{1}{m^q}\int_{{\rm G}(n,m)}R_m({((1-\lambda)\rho_K^p+\lambda\rho_L^p)^{\frac{m}{p}}})^q\,d\nu_{n,m}\\
 \nonumber
 &\geq\frac{1}{m^q}\int_{{\rm G}(n,m)}R_m({((1-\lambda)\rho_K^m+\lambda\rho_L^m)})^q\,d\nu_{n,m}\\
 \nonumber
 &=\frac{1}{m^q}\int_{{\rm G}(n,m)}((1-\lambda)R_m(\rho_K^m)+\lambda R_m(\rho_L^m))^q\,d\nu_{n,m}\\
 \nonumber
 &\geq\frac{1}{m^q}\int_{{\rm G}(n,m)}(1-\lambda)R_m(\rho_K^m)^q+\lambda R_m(\rho_L^m)^q\,d\nu_{n,m}\\
 \nonumber
 &=(1-\lambda)\widetilde{\Psi}_{m,q}(K)+\lambda \widetilde{\Psi}_{m,q}(L)=1.
 \end{align}
If $\widetilde{\Psi}_{m,q}((1-\lambda)\cdot K+_p\lambda\cdot L)=1$, then equality in Lemma~\ref{lp-include} (cf. \eqref{lp-include-qpos}) yields that $K$ and $L$ are dilates, and in turn $\widetilde{\Psi}_{m,q}(K)=\widetilde{\Psi}_{m,q}(L)=1$ implies that $K=L$. Therefore, if $K,L\in\mathcal{K}^n_{(o)}$, then we deduce \eqref{Psimqp-Minkowski-q01-eq} and 
\eqref{Psimqp-BM-q01-eq} together with the characterization of equality by 
Proposition~\ref{LpBMMequivalent}.

The inequality \eqref{Psimqp-BM-q01-eq} for $K,L\in\mathcal{K}^n_{o}$ follows by approximation and the continuity of $\widetilde{\Psi}_{m,q}(\cdot)$. The equality case of \eqref{Psimqp-BM-q01-eq} for $K,L\in\mathcal{K}^n_{o}$ and $p>m$ can be  handled using Claim~\ref{Lp-Lprime} as in the proof of Proposition~\ref{Psimqp-qlarger1}. 
\end{proof}

\begin{prop}
\label{Psimqp-qneg}
Let $m\in\{1,\ldots,n-1\}$, $q<0$ and  $p\geq m$.
If $K,L\in\mathcal{K}^n_{(o)}$ and $\alpha,\beta>0$, then
 \begin{align}
 \label{Psimqp-Minkowski-qneg-eq}
\int_{S^{n-1}}h_L^p\,d\widetilde{A}_{m,q,p}(K,\cdot)\geq  &m\widetilde{\Psi}_{m,q}(K)^{1-\frac{p}{mq}}\widetilde{\Psi}_{m,q}(L)^{\frac{p}{mq}},\\
\label{Psimqp-BM-qneg-eq}
\widetilde{\Psi}_{m,q}(\alpha\cdot K+_p\beta\cdot  L)^{\frac{p}{mq}}\geq &\alpha\widetilde{\Psi}_{m,q}(K)^{\frac{p}{mq}}+\beta\widetilde{\Psi}_{m,q}(L)^{\frac{p}{mq}},
\end{align}
where equality holds in either inequalities if and only if $K$ and $L$ are dilates.
\end{prop}
\begin{proof}
If $p\geq m$, $q<0$, $\lambda\in(0,1)$ and $K,L\in\mathcal{K}^n_{(o)}$ satisfy
$\widetilde{\Psi}_{m,q}(K)=\widetilde{\Psi}_{m,q}(L)=1$, then applying first
Lemma~\ref{lp-include}, secondly the Jensen inequality \eqref{Jensen} with $\theta=\frac{m}p$ and $\tau=1$, and after that with $\theta=1$ and $\tau=q$, and using that $q$ is negative, we deduce that
\begin{align}
 \nonumber
  \widetilde{\Psi}_{m,q}((1-\lambda)\cdot K+_p\lambda\cdot L)&=\frac{1}{m^q}\int_{{\rm G}(n,m)}R_m(\rho_{(1-\lambda)\cdot K+_p\lambda\cdot L}^m)^q\,d\nu_{n,m}\\
\label{lp-include-qneg}
 &\leq\frac{1}{m^q}\int_{{\rm G}(n,m)}R_m(\rho_{(1-\lambda)\cdot K\widetilde{+}_p\lambda\cdot L}^m)^q\,d\nu_{n,m}\\
 \nonumber
 &=\frac{1}{m^q}\int_{{\rm G}(n,m)}R_m({((1-\lambda)\rho_K^p+\lambda\rho_L^p)^{\frac{m}{p}}})^q\,d\nu_{n,m}\\
 \nonumber
 &\leq\frac{1}{m^q}\int_{{\rm G}(n,m)}R_m({((1-\lambda)\rho_K^m+\lambda\rho_L^m)})^q\,d\nu_{n,m}\\
 \nonumber
 &=\frac{1}{m^q}\int_{{\rm G}(n,m)}((1-\lambda)R_m(\rho_K^m)+\lambda R_m(\rho_L^m))^q\,d\nu_{n,m}\\
 \nonumber
 &\leq\frac{1}{m^q}\int_{{\rm G}(n,m)}(1-\lambda)R_m(\rho_K^m)^q+\lambda R_m(\rho_L^m)^q\,d\nu_{n,m}\\
 \nonumber
 &=(1-\lambda)\widetilde{\Psi}_{m,q}(K)+\lambda \widetilde{\Psi}_{m,q}(L)=1.
 \end{align}
If $\widetilde{\Psi}_{m,q}((1-\lambda)\cdot K+_p\lambda\cdot L)=1$, then equality in Lemma~\ref{lp-include} (cf. \eqref{lp-include-qneg}) yields that $K$ and $L$ are dilates, and in turn $\widetilde{\Psi}_{m,q}(K)=\widetilde{\Psi}_{m,q}(L)=1$ implies that $K=L$. Therefore, if $K,L\in\mathcal{K}^n_{(o)}$, then we deduce \eqref{Psimqp-Minkowski-qneg-eq} and 
\eqref{Psimqp-BM-qneg-eq} together with the characterization of equality by 
Proposition~\ref{LpBMMequivalent}.
\end{proof}

\begin{prop}
\label{Psimqp-qzero-prop}
If $m\in \left\{1,\ldots,n-1\right\}$, $q=0$, $p\geq m$, $\lambda\in(0,1)$ and $K,L\in\mathcal{K}^n_{(o)}$, then
\begin{align}
\label{Psimqp-BM-qzero-prop}
\widetilde{\Psi}_{m,0}((1-\lambda)\cdot K+_p\lambda\cdot  L)\geq &(1-\lambda)\widetilde{\Psi}_{m,0}(K)+\lambda\widetilde{\Psi}_{m,0}(L),\\
\label{Psimqp-Minkowski-qzero-prop}
\int_{S^{n-1}}h_L^p\,d\widetilde{A}_{m,0,p}(K,\cdot)\geq & p\left(\widetilde{\Psi}_{m,0}(L)-\widetilde{\Psi}_{m,0}(K)\right)+m,
\end{align}
and equality holds either in \eqref{Psimqp-BM-qzero-prop} or in \eqref{Psimqp-Minkowski-qzero-prop}  if and only if $K=L$.
\end{prop}
\begin{proof}
The Remark after Lemma~\ref{LpBMMequivalent} says that it is sufficient to verify
\eqref{Psimqp-BM-qzero-prop}. Again, applying first
Lemma~\ref{lp-include}, secondly the Jensen inequality \eqref{Jensen} with $\theta=\frac{m}p$ and $\tau=1$, and after that the concavity of the $\log$ function, we deduce 
\begin{align}
\nonumber
 &\widetilde{\Psi}_{m,0}((1-\lambda)\cdot K+_p\lambda\cdot L)\\
 \nonumber
 =&\int_{{\rm G}(n,m)}\log R_m\left(\frac{1}{m}\rho_{(1-\lambda)\cdot K+_p\lambda\cdot L}^m\right)\,d\nu_{n,m}\\
 \nonumber
 =&-\int_{{\rm G}(n,m)}\log m\,d\nu_{n,m}+\int_{{\rm G}(n,m)}\log R_m(\rho_{(1-\lambda)\cdot K+_p\lambda\cdot L}^m)\,d\nu_{n,m}\\
 \label{lp-include-qzero}
 \geq&-\log m+\int_{{\rm G}(n,m)}\log R_m(\rho_{(1-\lambda)\cdot K\widetilde{+}_p\lambda\cdot L}^m)\,d\nu_{n,m}\\
\nonumber
 =&-\log m+\int_{{\rm G}(n,m)}\log R_m({((1-\lambda)\rho_K^p+\lambda\rho_L^p)^{\frac{m}{p}}})\,d\nu_{n,m}\\
 \nonumber
 \geq&-\log m+\int_{{\rm G}(n,m)}\log R_m({((1-\lambda)\rho_K^m+\lambda\rho_L^m)})\,d\nu_{n,m}\\
 \nonumber
 =&-\log m+\int_{{\rm G}(n,m)}\log((1-\lambda) R_m(\rho_K^m)+\lambda R_m\rho_L^m)\,d\nu_{n,m}\\
 \label{log-is-concave}
 \geq&-\log m+(1-\lambda) \int_{{\rm G}(n,m)}\log(R_m(\rho_K^m)\,d\nu_{n,m}+\\
 \nonumber
 &+\lambda\int_{{\rm G}(n,m)}\log R_m(\rho_L^m)d\,d\nu_{n,m}\\
 \nonumber
 =&(1-\lambda)\widetilde{\Psi}_{m,0}(K)+\lambda \widetilde{\Psi}_{m,0}(L).
 \end{align}
 If equality holds in \eqref{Psimqp-BM-qzero-prop}, then equality in Lemma~\ref{lp-include} (cf. \eqref{lp-include-qzero}) implies that $K$ and $L$ are dilates, 
 thus equality in \eqref{log-is-concave} yields $K=L$.
\end{proof}

We are ready to discuss the uniqueness of $L_p$ centro-section Minkowski measure.

\begin{theorem}
\label{uniqueness-general}
Let $m\in \left\{1,\ldots,n-1\right\}$, $q\in\R$, and let $p>mq$ if $q\geq 1$, and $p\geq m$ if $q<1$.
If $\widetilde{A}_{m,q,p}(K,\cdot)=\widetilde{A}_{m,q,p}(L,\cdot)$ holds for $K,L\in \mathcal{K}^n_{(o)}$, then  $K=L$.
 \end{theorem}
\begin{proof}
First, we assume that $q\neq 0$.
It follows from \eqref{Psimqp-Minkowski-qlarger1-eq}, \eqref{Psimqp-Minkowski-q01-eq},
\eqref{Psimqp-Minkowski-qneg-eq}
and $\widetilde{A}_{m,q,p}(K,\cdot)=\widetilde{A}_{m,q,p}(L,\cdot)$ that
\begin{align}
\nonumber
m\widetilde{\Psi}_{m,q}(L)=&\int_{S^{n-1}}h_L^p(v)d\widetilde{A}_{m,q,p}(L,v)=\int_{S^{n-1}}h_L^p(v)d\widetilde{A}_{m,q,p}(K,v)\\
\label{uniqueness-inequality-qnonzero}
&\geq m\widetilde{\Psi}_{m,q}(K)^{1-\frac{p}{mq}}\widetilde{\Psi}_{m,q}(L)^{\frac{p}{mq}},
\end{align}
and hence $\widetilde{\Psi}_{m,q}(L)^{1-\frac{p}{mq}}\geq \widetilde{\Psi}_{m,q}(K)^{1-\frac{p}{mq}}$. Interchanging the role of $L$ and $K$ yields that $\widetilde{\Psi}_{m,q}(K)^{1-\frac{p}{mq}}\geq \widetilde{\Psi}_{m,q}(L)^{1-\frac{p}{mq}}$; therefore, $\widetilde{\Psi}_{m,q}(L)=\widetilde{\Psi}_{m,q}(K)$ holds 
by $1-\frac{p}{mq}\neq 0$. It follows that equality holds in \eqref{uniqueness-inequality-qnonzero}, and in turn the relevant inequality out of \eqref{Psimqp-Minkowski-qlarger1-eq}, \eqref{Psimqp-Minkowski-q01-eq},
\eqref{Psimqp-Minkowski-qneg-eq}. We conclude that $K$ and $L$ are dilates, and hence $K=L$ follows $\widetilde{\Psi}_{m,q}(K)=\widetilde{\Psi}_{m,q}(L)$ and the homogeneity of $\widetilde{\Psi}_{m,q}(\cdot)$.

If $q=0$, then \eqref{Psimqp-Minkowski-qzero-prop} yields that
\begin{align}
\nonumber
m=&\int_{S^{n-1}}h_L^p(v)d\widetilde{A}_{m,0,p}(L,v)=\int_{S^{n-1}}h_L^p(v)d\widetilde{A}_{m,0,p}(K,v)\\
\label{uniqueness-inequality-qzero}
&\geq p\left(\widetilde{\Psi}_{m,0}(L)-\widetilde{\Psi}_{m,0}(K)\right)+m,
\end{align}
and hence $\widetilde{\Psi}_{m,0}(L)\leq \widetilde{\Psi}_{m,0}(K)$. 
Interchanging the role of $K$ and $L$ implies that $\widetilde{\Psi}_{m,0}(L)\geq \widetilde{\Psi}_{m,0}(K)$, thus $\widetilde{\Psi}_{m,0}(L)= \widetilde{\Psi}_{m,0}(K)$. It follows that equality holds in \eqref{uniqueness-inequality-qzero}; therefore, the equality condition for \eqref{Psimqp-Minkowski-qzero-prop} yields $K=L$.
\end{proof}

For $q\neq 0$, our methods yield the $L_p$ centro-sectional Minkowski and Brunn-Minkowski inequalities for origin symmetric convex bodies for any $p>\max\{0,mq\}$, not only for $p>\max\{m,mq\}$. We write $\mathcal{K}^n_{e}$ to denote the family of $o$-symmetric convex bodies in $\R^n$.

\begin{theorem}
\label{Psimqp-even}
Let $m\in\{1,\ldots,n-1\}$, $q\neq 0$ and  $p>\max\{0,mq\}$.
If $K,L\in\mathcal{K}^n_{e}$ and $\alpha,\beta>0$, then
 \begin{align}
 \label{Psimqp-Minkowski-even-eq}
\int_{S^{n-1}}h_L^p\,d\widetilde{A}_{m,q,p}(K,\cdot)\geq  &m\widetilde{\Psi}_{m,q}(K)^{1-\frac{p}{mq}}\widetilde{\Psi}_{m,q}(L)^{\frac{p}{mq}},\\
\label{Psimqp-BM-even-eq}
\widetilde{\Psi}_{m,q}(\alpha\cdot K+_p\beta\cdot  L)^{\frac{p}{mq}}\geq &\alpha\widetilde{\Psi}_{m,q}(K)^{\frac{p}{mq}}+\beta\widetilde{\Psi}_{m,q}(L)^{\frac{p}{mq}},
\end{align}
where equality holds in either inequalities if and only if $K$ and $L$ are dilates.

In addition, $\widetilde{A}_{m,q,p}(K,\cdot)=\widetilde{A}_{m,q,p}(L,\cdot)$ if and only if $K=L$.
\end{theorem}
\begin{proof}
First, we prove that if  $K\in\mathcal{K}^n_{e}$ has $C^{2,\alpha}_+$ boundary for some $\alpha\in(0,1)$, then \eqref{Psimqp-Minkowski-even-eq} holds for any $L\in\mathcal{K}^n_{e}$. In this case, $d\widetilde{A}_{m,q,p}(K,\cdot)=f\,d\HH^{n-1}$ for a positive function $f\in C^{0,\alpha}(S^{n-1}$.
Therefore, the argument leading to \eqref{Psimqp-Minkowski-even-eq} is exactly the same as in Lemma~\ref{polytope-Minkowski}, only using Proposition~\ref{extremal-problem-even}
 instead of Proposition~\ref{extremal-problem}, and Proposition~\ref{uniqueness-p-mq}  instead of Proposition~\ref{discrete-unique}.

Since $o$-symmetric convex bodies with $C^{2,\alpha}_+$ boundary are dense in $\mathcal{K}^n_{e}$, and
$\widetilde{A}_{m,q,p}(K,\cdot)$ is weakly continuous on $\mathcal{K}^n_{e}$ (cf. Proposition~\ref{Amqp-weak-convergence}), approximation yields \eqref{Psimqp-Minkowski-even-eq}  for any $K,L\in\mathcal{K}^n_{e}$ (without the characterization of equality).
Now, \eqref{Psimqp-BM-even-eq} follows by Lemma~\ref{LpBMMequivalent} (without the characterization of equality).

Finally, the characterization of equality in \eqref{Psimqp-Minkowski-even-eq} and \eqref{Psimqp-BM-even-eq} can be done using Lemma~\ref{LpBMMequivalent} and Claim~\ref{Lp-Lprime} as in Proposition~\ref{Psimqp-qlarger1}, and for the $L_p$ centro-sectional Minkowski problem as in 
Theorem~\ref{uniqueness-general}.
\end{proof}

\begin{proof}[Proofs of 
Theorem~\ref{Psimqp-uniqueness},
Theorem~\ref{Psimqp-BM} and
Theorem~\ref{symmetric}]
Theorem~\ref{uniqueness-general} directly yields Theorem~\ref{Psimqp-uniqueness}, and combining
 Proposition~\ref{Psimqp-qlarger1},
Proposition~\ref{Psimqp-q01} and
Proposition~\ref{Psimqp-qneg} implies
Theorem~\ref{Psimqp-BM}.
Finally, the uniqueness part of Theorem~\ref{symmetric} follows from Theorem~\ref{Psimqp-even}.    
\end{proof}

\section{Acknowledgements}
K\'aroly J. B\"or\"oczky   is supported by the NKKP Advanced grant 150613, and Jiaqian Liu is supported by  the National Natural Science Foundation of China (12401252).

\end{document}